\documentclass[a4paper,11pt]{amsart}

\usepackage{latexsym}
\usepackage{amssymb,amsmath}
\usepackage{graphics}
\usepackage{url,color}
\usepackage[vcentermath]{youngtab}
  
\usepackage{booktabs}  
\usepackage{tikz}
\usetikzlibrary{calc}
\usetikzlibrary{arrows.meta}
\newcommand*{\symdiff}{\,\scalebox{0.75}{$\bigtriangleup$}\,}

\newcommand*{\ssymdiff}{\hskip0.5pt\scalebox{0.5}{$\bigtriangleup$}\hskip0.5pt}

\renewcommand{\i}{\mathrm{i}}
\usepackage{caption}
\newtheorem{theorem}{Theorem}[section]
\newtheorem*{theorem*}{Theorem}
\newtheorem{corollary}[theorem]{Corollary}

\newtheorem{problem}[theorem]{Problem}
\newtheorem{proposition}[theorem]{Proposition}
\newtheorem{conjecture}[theorem]{Conjecture}

\newtheorem{lemma}[theorem]{Lemma}

\theoremstyle{definition}

\newtheorem{remark}[theorem]{Remark}
\newtheorem{example}[theorem]{Example}

\newcommand{\N}{\mathbb{N}}
\newcommand{\Z}{\mathbb{Z}}

\newcommand{\F}{\mathbb{F}}

\renewcommand{\epsilon}{\varepsilon}

\DeclareMathOperator{\Hom}{Hom}
\DeclareMathOperator{\End}{End}

\DeclareMathOperator{\rad}{rad}
\DeclareMathOperator{\im}{im}

\DeclareMathOperator{\sgn}{sgn}

\newcommand{\Res}{\big\downarrow}

\newcommand{\res}{\!\downarrow}

\renewcommand{\theta}{\vartheta}

\newcounter{thmlistcnt}
\newenvironment{thmlist}%
	{\setcounter{thmlistcnt}{0}%
	\begin{list}{\emph{(\roman{thmlistcnt})}}{%
		\usecounter{thmlistcnt}%
		\setlength{\topsep}{0pt}%
		\setlength{\leftmargin}{6pt}%
		\setlength{\itemsep}{0pt}%
		\setlength{\labelwidth}{17pt}
		\setlength{\itemindent}{30pt}}%
	}%
	{\end{list}}%

\newcounter{thmlistaltcnt}
\newenvironment{thmlistalt}%
	{\setcounter{thmlistaltcnt}{0}%
	\begin{list}{\emph{(\roman{thmlistaltcnt})}}{%
		\usecounter{thmlistaltcnt}%
		\setlength{\topsep}{2pt}%
		\setlength{\leftmargin}{27pt}%
		\setlength{\itemsep}{0pt}%
		\setlength{\labelwidth}{18pt}
		\setlength{\itemindent}{0pt}}%
	}%
	{\end{list}}%

\newcounter{thmlistaltlesstopcnt}
\newenvironment{thmlistaltlesstop}%
	{\setcounter{thmlistaltlesstopcnt}{0}%
	\begin{list}{\emph{(\roman{thmlistaltlesstopcnt})}}{%
		\usecounter{thmlistaltlesstopcnt}%
		\setlength{\topsep}{0pt}%
		\setlength{\leftmargin}{24pt}%
		\setlength{\itemsep}{0pt}%
		\setlength{\labelwidth}{18pt}
		\setlength{\itemindent}{0pt}}%
	}%
	{\end{list}}%
	
\let\phic\varphi
\renewcommand{\phi}[1]{{\phic^{(#1)}}}
\let\Omegac\Omega
\renewcommand{\Omega}[1]{\Omegac_{#1}}
\newcommand{\Omegan}[2]{\Omegac_{#1}^{[#2]}}

\let\alphac\alpha
\renewcommand{\alpha}[1]{{\alphac^{(#1)}}}

\def\sf{0.9}
\newcommand{\vk}[2]{v^{\scalebox{\sf}{$\scriptstyle{(#1)}$}}_{#2}}
\newcommand{\phik}[2]{\phic^{\scalebox{\sf}{$\scriptstyle{(#1)}$}}_{#2}\hskip-0.5pt} 
\newcommand{\phikn}[3]{\phic^{\scalebox{\sf}{$\scriptstyle{(#1)}[#3]$}}_{#2}} 

\newcommand{\phikstar}[2]{{\phik{#1}{#2}}^\star}
\newcommand{\phikstarkt}{\phic^{\scalebox{\sf}{$\scriptstyle{(t)}^\star$}}_{k+t}}

\newcommand{\phio}[2]{\phik{#1}{#2}} %for more indices change to phik
\newcommand{\phioo}[2]{\phi{#1}} 

\newcommand{\tauperm}{\sigma}
\newcommand{\sigmapower}{\rho}

\renewcommand{\c}{\ell}

\newcommand{\K}{\mathbb{F}}
\newcommand{\HR}{\bar{H}}


\begin{document}
\title[Multistep homology of the simplex]{The multistep homology
of the simplex and representations of symmetric groups}
\date{\today}
\author{Mark Wildon}

\maketitle
\thispagestyle{empty}

\begin{abstract}
The symmetric group on a set acts transitively on its subsets of a given size.
We define homomorphisms between the corresponding permutation modules, defined
over a field of characteristic two, which generalize the 
boundary maps from simplicial homology.
The main results determine when these chain complexes are exact and when they are split exact. 
As a corollary we obtain a new explicit construction of the basic spin modules for
the symmetric group. % in characteristic two.
\end{abstract}

\section{Introduction}

Fix $n \in \N$ and let $S_n$ denote the symmetric group of degree $n$.
For each~$k \in \Z$, let $\Omega{k}$
denote the  set of all $k$-subsets of $\{1,\ldots, n\}$, permuted by the action of $S_n$.
Let $\F$ be a field and let
$\F\Omega{k}$ be the $\F$-vector space of all formal $\F$-linear combinations of the
elements of $\Omega{k}$. Thus $\F\Omega{k}$ is an $\F S_n$-module of dimension $\binom{n}{k}$
having~$\Omega{k}$ as a permutation basis.
For instance if $n \ge 5$ then $\{1,2,3\} + \{3,4,5\} \in \F\Omega{3}$ is
sent to $\{1,2,3\} + \{1,4,5\}$ by the transposition swapping~$1$ and $3$.

Given $t \in \N_0$ and $k \in \Z$, let %NB need \N_0 for final summands in Suspension Lemma
$\phik{t}{k} : \F\Omega{k} \rightarrow \F\Omega{k-t}$
be the  $\F S_n$-module homomorphism defined on each $Y \in \Omega{k}$ by 
\begin{equation}\label{eq:phi} 
Y \phik{t}{k} = \sum_{X \subseteq Y \atop |X| = |Y| - t\rule{0pt}{5.25pt}} X.
\end{equation}
(Throughout we work with right-modules and write maps on the right.)
Motivated by the connection with simplicial homology discussed below,
we call $\phik{t}{k}$ a \emph{multistep boundary map}.
This article concerns the remarkably intricate behaviour of the multistep boundary
 maps when $\F$ has characteristic~two.

\enlargethispage{4pt}
Given $Z \in \Omega{k}$ and $t \in \N$, we may compute
$Z \phik{t}{k} \phik{t}{k-t}$  by summing over all chains $Z \supseteq Y \supseteq X$
with $Y \in \Omega{k-t}$ and $X \in \Omega{k-2t}$. For each~$X$ there are $\binom{2t}{t}$ choices for $Y$;
since $\binom{2t}{t} \equiv 0$ mod $2$, and
 $\F$ has characteristic two,  $X\phik{t}{k}\phik{t}{k-t} = 0$.
Hence if $a < t$ and $c \in \N_0$ is maximal such that $a + ct \le n$ then
\begin{equation}
\label{eq:phiComplex} \scalebox{0.975}{$\displaystyle
0 \rightarrow \F\Omega{a+ct} \xrightarrow{\phik{t}{a+ct}} \F\Omega{a+(c-1)t} 
\xrightarrow{\phik{t}{a+(c-1)t}} \cdots \xrightarrow{\phik{t}{a+2t}}
\F\Omega{a+t} \xrightarrow{\phik{t}{a+t}} \F\Omega{a} \rightarrow 0$} \end{equation}
is a chain complex of  $\F S_n$-modules, each non-zero except at the beginning and end.
% for every $a \in \Z$ and $t \in \N$.
Its \emph{homology} in \emph{degree} $k$ is, by definition, 
the $\F S_n$-module $\ker \phik{t}{k} / \im \phik{t}{k+t}$. %Let $H_k$ denote this module.

If $t=1$ then the chain complex~\eqref{eq:phiComplex} is exact in every degree.
%Equivalently, as is well-known,
%the solid $(n-1)$-simplex has zero homology with coefficients in~$\F$ in all
%non-zero dimensions.
 Moreover~\eqref{eq:phiComplex} is \emph{split exact}, in the sense that,
for each $k$, there is an $\F S_n$-submodule $C_k$ of $\F \Omega{k}$ such that \smash{$\F\Omega{k} 
= \ker \phik{1}{k}\oplus C_k$},
if and only if~$n$ is odd. We give short proofs
of these  results %  and the  isomorphism $\ker \phik{1}{k} \cong S^{(n-k,1^k)}$ 
in \S\ref{sec:background} below.

Our first main theorem gives a complete description of the homology modules when $t=2$.
The following notation is required:
for $k$ such that $2k \le n$,
define $G_{k-1} = \bigl\langle (1, 2) \bigl\rangle \times \cdots \times 
\bigl\langle \bigl( 2(k-1)-1, 2(k-1)\bigr) \bigr\rangle$
and 
\[ v_k = \{2,4,\ldots, 2k\} \sum_{\tauperm \in G_{k-1}} \tauperm. \]
(These elements are illustrated in Example~\ref{ex:ex}.)
Let
$D^{(n-k,k)}$ denote the simple $\F S_n$-module defined, with its usual definition,
in \S\ref{sec:prelim} below.

\begin{theorem}\label{thm:epsilonHomology} Let $\epsilon_k : \F\Omega{k}\rightarrow
\F\Omega{k-2}$ denote the two-step boundary map $\phik{2}{k}\hskip-0.5pt$,  as defined in~\eqref{eq:phi}, and let $H_k = 
\ker \epsilon_k / \im \epsilon_{k+2}$. Then
\[ 
H_k
 \cong \begin{cases} E^{(m+1,m-1)} & \text{if $n = 2m$ is even and $k = m$} \\
D^{(m+1,m)} & \text{if $n = 2m+1$ is odd and $k = m$ or $k = m+1$} \\
0           & \text{otherwise,} \end{cases} \]
where 
$E^{(m+1,m-1)}$ is a non-split extension %of $F$ by itself if $m=1$ or of 
of $D^{(m+1,m-1)}$ by itself. Moreover, if $n = 2m$ or $n=2m+1$ then
$H_m$ is the submodule of $\F\Omega{m} / \im \epsilon_{m+2}$ generated by $v_m  + \im \epsilon_{m+2}$
and, for each $m \in \N$, %$\dim D^{(m+1,m-1)} = 2^{m-1}$,
there are isomorphisms $D^{(m+1,m)}\!\res_{S_{2m}}\, \cong E^{(m+1,m-1)}$, $D^{(m+1,m-1)}\!\res_{S_{2m-1}}\,
\cong D^{(m,m-1)}$. 

%\vspace*{-2.5pt}
%\begin{thmlistalt}
%\item 
%if $n = 2m$ is even then
%$H_m$ is the submodule of $\F\Omega{m} / \im \epsilon_{m+2}$ generated by $v_m + \im \epsilon_{m+2}$
%and the map $\theta : H_m \rightarrow H_m$ defined by $v_m \mapsto v_m + v_m (2m-1,2m) +  \im \epsilon_{k+2}$
%is a non-zero endomorphism of $H_m$ such that $\theta^2 = 0$;
%
%\item 
%if $n=2m+1$ is odd then
%$H_m$ is 
%\end{thmlistalt}
\end{theorem}

The results on the restrictions of $D^{(m+1,m-1)}$ and $D^{(m+1,m)}$
 in Theorem~\ref{thm:epsilonHomology} are 
originally due to Danz and K{\"u}lshammer \cite[Proposition~3.3]{DanzKulshammer}; 
they are included so that
the theorem can be proved by induction as it is stated.
In Corollary~\ref{cor:nsplit} we take $n=2m$ and construct an $\F S_{2m}$-endomorphism $\theta$
of $H_m$ such that $\theta$ is non-zero and $\theta^2 = 0$,
making explicit the structure of the non-split extension
$E^{(m+1,m-1)}$.

In particular, Theorem~\ref{thm:epsilonHomology} implies that
 the chain complex of $\F S_{2m}$-modules
\[ 0 \rightarrow \F\Omega{2m} \xrightarrow{\epsilon_{2m}} \F\Omega{2m-2} \xrightarrow{\epsilon_{2m-2}} \cdots 
\xrightarrow{\epsilon_{4}} \F\Omega{2} \xrightarrow{\epsilon_{2}}
\F\Omega{0} \rightarrow 0 \]
is exact whenever $m$ is odd; if  $m$ is even then it has non-zero homology of $E^{(m+1,m-1)}$ uniquely in
degree $m$. This categorifies the binomial identity
\begin{equation} \label{eq:evenBinom} \sum_{j=0}^m (-1)^j \binom{2m}{2j} = \begin{cases} (-1)^{m/2} 2^{m} & \text{if $m$ is even} \\
0 & \text{if $m$ is odd}. \end{cases} \end{equation}
%The analogous identify for odd $n$ is used to prove Corollary~\ref{cor:dimensions} below.

Our second main theorem
determines the degrees in which the chain complex~\eqref{eq:phiComplex} is exact.
In particular, case (ii) 
determines when one of the maps is surjective or injective.

\begin{theorem}\label{thm:gen}
Let $t \in \N$, let $n \in \N$ and let $0 \le k \le n$.
 Let $2^\tau$ be the least two-power appearing
in the binary form of $t$. The sequence

\vspace*{-14.5pt}
\begin{equation}
\label{eq:seq} 
\F\Omega{k+t} \xrightarrow{\phik{t}{k+t}} \F\Omega{k} \xrightarrow{\phik{t}{k}} \F\Omega{k-t} \end{equation}
is exact if and only if one of 
\begin{thmlist}
\item
$t=1$; 
\item 
$k < 2^\tau$ and $k+t \le n-k$ \emph{or} $n-k < 2^\tau$ and $n-k+t \le k$;
\item 
$t$ is a two-power \emph{and} $n \ge 2k+t$ or $n \le 2k-t$.
\end{thmlist}
\end{theorem}

We also characterize when~\eqref{eq:phiComplex} is exact in every degree. It seems remarkable that
this is the case if and only if it is split exact in every degree.

\begin{theorem}\label{thm:splitExact}
Let $2^\tau$ be the least two-power appearing in the binary form of $t$.
The chain complex~\eqref{eq:phiComplex} is exact in every degree if and only if one of
\begin{thmlist}
\item[\emph{(a)}] $n = 2a + t$ and $a < 2^\tau$;
\item[\emph{(b)}] $t$ is a two-power and $n \equiv 2a + t$ mod $2t$.
\end{thmlist}
Moreover, if either \emph{(a)} or \emph{(b)} holds then~\eqref{eq:phiComplex}
is split exact in every degree.
\end{theorem}

We end this introduction with two examples showing some of the rich behaviour of the kernels and
images of the multistep boundary maps. %, extending the example following Lemma 2.7 in \cite{JamesChar2}.
For readability we write $\gamma_k$ for $\phik{1}{k}$.

%\begin{example}\label{ex:extra}
%When $n=4$ the Loewy layers of the modules in the exact
%chain complex $\F\Omega{4} \xrightarrow{\gamma_4} \F\Omega{3} \xrightarrow{\gamma_3} 
%\F\Omega{2} \xrightarrow{\gamma_2}
%\F\Omega{1} \xrightarrow{\gamma_1} \F\Omega{0}$ are
%shown below.
%\[ \F \xrightarrow{\gamma_4} \begin{matrix} \F \\ D^{(3,1)} \\ \F \end{matrix}
%\xrightarrow{\gamma_3}\begin{matrix} \F \oplus D^{(3,1)} \\ \F \oplus D^{(3,1)}\end{matrix} \xrightarrow{\gamma_2} \begin{matrix} \F \\ D^{(3,1)} \\ \F \end{matrix} 
%\xrightarrow{\gamma_1} \F . \]
%\end{example}

\begin{example}\label{ex:ex}
When $n=6$ the Loewy layers of the modules
in the exact chain complex %$0 \rightarrow 
$\F\Omega{6} \xrightarrow{\gamma_6} \F\Omega{5} \xrightarrow{\gamma_5} 
\cdots \xrightarrow{\gamma_2}
\F\Omega{1} \xrightarrow{\gamma_1} \F\Omega{0}$ % \rightarrow 0$ 
are shown below.

\vskip6pt
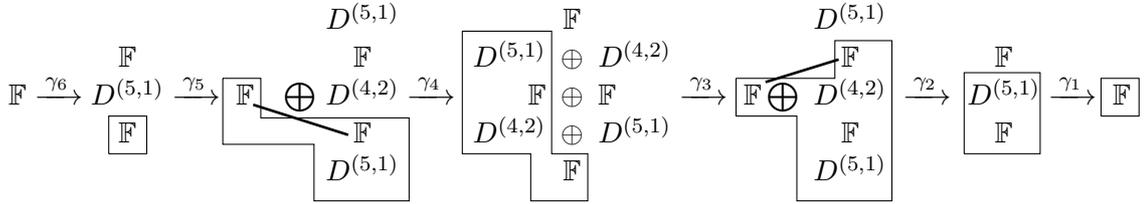
\hspace*{-0.52in}\begin{tikzpicture}[x=0.5cm, y=0.5cm]
\node[right] (O) at (0,0) 
{$\F \xrightarrow{\,\gamma_6\,} \begin{matrix} \F \\ D^{(5,1)} \\ \F \end{matrix}
     \xrightarrow{\,\gamma_5\,}
     \; \F\;\;\; \bigoplus \,
       \begin{matrix} D^{(5,1)} \\ \F \\ D^{(4,2)} \\ \F \\ D^{(5,1)} \end{matrix}
      \xrightarrow{\,\gamma_4\,}
       \ \begin{matrix} \F \\ D^{(5,1)} \;\oplus\; D^{(4,2)} \\ \F \;\oplus\; \F \\
                                    D^{(4,2)} \;\oplus\; D^{(5,1)} \\ \F \end{matrix}
      \xrightarrow{\,\gamma_3\,} \;  \F \bigoplus \
       \begin{matrix} D^{(5,1)} \\ \F \\ D^{(4,2)} \\ \F \\ D^{(5,1)}\end{matrix}
     \; \xrightarrow{\,\gamma_2\,}  \; \begin{matrix} \F \\ D^{(5,1)} \\ \F \end{matrix}
      \xrightarrow{\,\gamma_1\,}\; \F$}; 

\draw(2.9,-1.5)--(3.9,-1.5)--(3.9,-0.5)--(2.9,-0.5)--(2.9,-1.5);
\draw(5.9,0.5)--(6.9,0.5)--(6.9,-0.55)--(10.8,-0.55)--(10.8,-2.75)--(8.3,-2.75)--(8.3,-1.25)--(5.9,-1.25)--(5.9,0.5);
\draw[line width = 1pt] (6.7,-0.2)--(9.2,-1.0);
\draw(12.2,1.75)--(14.55,1.75)--(14.55,-1.5)--(15.5,-1.5)--(15.5,-2.75)--(14.0,-2.75)--(14.0,-1.5)--(12.2,-1.5)--(12.2,1.75);
\draw(19.4,0.5)--(22.0,0.5)--(22.0,1.5)--(23.5,1.5)--(23.5,-2.75)--(21.0,-2.75)--(21.0,-0.5)--(19.4,-0.5)--(19.4,0.5);
\draw[line width = 1pt] (20.175,0.4)--(22.1,1);
\draw(25.4,-1.5)--(27.4,-1.5)--(27.4,0.65)--(25.4,0.65)--(25.4,-1.5); %--(3.9,-0.5)--(2.9,-0.5)--(2.9,-1.5);
\draw(29.0,-0.5)--(30.0,-0.5)--(30.0,0.5)--(29.0,0.5)--(29.0,-0.5); %--(3.9,-0.5)--(2.9,-0.5)--(2.9,-1.5);

\end{tikzpicture}

\vskip7.5pt

As predicted by Theorem~\ref{thm:epsilonHomology}, $\ker \epsilon_4 \cong \F$ is a direct
summand of $\F\Omega{4}$ and $\ker \epsilon_2$ is the (unique) co-dimension $1$ direct summand
of $\F\Omega{2}$. Thus
the chain complex $0 \rightarrow \F\Omega{6} \xrightarrow{\epsilon_6}
\F\Omega{4} \xrightarrow{\epsilon_4} \F\Omega{2} \xrightarrow{\epsilon_2} \F\Omega{0}
\rightarrow 0$ is split exact.
Moreover $0 \rightarrow \F\Omega{5} \xrightarrow{\epsilon_5} \F\Omega{3} \xrightarrow{\epsilon_3}
\F\Omega{1} \rightarrow 0$ is exact except in degree~$3$, where it has homology $E^{(4,2)}$.
By Theorem~\ref{thm:epsilonHomology} the homology is generated by
$v_3 + \im \epsilon_5$, where $v_3 = \{2,4,6\} + \{1,4,6\} + \{2,3,6\} + \{1,3,6\}$.

The boxes show the kernels of the maps $\gamma_k$.
For example, by Theorem~\ref{thm:gen}(i), $\ker \gamma_2$ is generated by $\{1,2,3\}\gamma_3 = \{1,2\} + \{2,3\} + \{3,1\}$.
Since $\ker \epsilon_2 = \langle X + Y : X, Y \in \Omega{2} \rangle$, 
the intersection
$\ker \gamma_2 \cap \ker \epsilon_2$ is generated by $\{1,2,3\}\gamma_3 + \{1,2,4\}\gamma_3 = \{1,3\} + \{2,3\} + \{1,4\} + \{2,4\}$; it is isomorphic to the Specht module $S^{(4,2)}$
and has composition factors $D^{(4,2)}$, $\F$, $D^{(5,1)}$.
It follows that
$\ker \gamma_2$ is not contained in either direct summand of $\F\Omega{2}$.
The line on the diagram above indicates a `diagonally embedded' submodule; this submodule is unique
if and only if $|\mathbb{F}_2| = 2$. The dual situation arises for $\ker \gamma_4$ and $\F\Omega{4}$.

It is an amusing exercise to show that the outer automorphism of $S_6$ 
swaps the simple modules $D^{(4,2)}$ and $D^{(5,1)}$ and leaves $\F\Omega{3}$ invariant.
In particular, applying it to the homology module $\ker \epsilon_3 / \im \epsilon_5 \cong E^{(4,2)}$
gives a non-split extension of $D^{(5,1)}$ by itself. 
\end{example}

\begin{remark}
In \S\ref{sec:background}  we show
that \smash{$\ker \phik{1}{k}$} is isomorphic to the Specht module \smash{$S^{(n-k,1^k)}$}, by
an explicit isomorphism defined on a generator for \smash{$\im \phik{t}{k+t}$}.
For small $k$, there are some interesting isomorphisms between
the kernels of the multistep boundary maps and Young modules. For example,
it follows from Proposition~\ref{prop:splitExact} that $\ker \epsilon_2 \cong Y^{(n-2,2)}$
whenever $n \equiv 2$ mod~$4$; Example~\ref{ex:ex} shows the case $n=6$.
In general, however,
$\smash{\ker \phik{t}{k}}$ appears
 to have no more explicit description than that given in the main theorems.
\end{remark}

The second example shows that~\eqref{eq:seq} may be split exact in cases when
the full chain complex~\eqref{eq:phiComplex} containing it fails even to be exact.

\begin{example}\label{ex:splitSurprise}
Take $n=13$. When $t=4$ and $a=0$, the chain complex~\eqref{eq:phiComplex}~is

\vspace*{-9pt}
\[ 0 \rightarrow \F\Omega{12} \xrightarrow{\phik{4}{12}} \F\Omega{8} \xrightarrow{\phik{4}{8}} \F\Omega{4} \xrightarrow{\phik{4}{4}} \F\Omega{0} \rightarrow 0. \]
Since $\binom{13}{4}$ is odd, the trivial module is a direct summand of $\F \Omega{4}$;
since $\ker \phik{4}{4} = \langle X + Y : X,\, Y \in \Omega{4} \rangle$, 
we have $\F\Omega{4} = \ker \phik{4}{4} \oplus \langle
\sum_{X \in \Omega{4}} \hskip-0.5pt X \rangle$. By Theorem~\ref{thm:gen}, \smash{$\ker \phik{4}{4} = \im \phik{4}{8}$}.
Therefore $\F\Omega{8} \rightarrow \F\Omega{4} \twoheadrightarrow \F\Omega{0}$ is split exact.
But, again by Theorem~\ref{thm:gen}, $\F\Omega{12} \hookrightarrow \F\Omega{8} \rightarrow \F\Omega{4}$
is not exact; the proof of Lemma~\ref{lemma:cfobstruction} shows that
the homology module $\ker \phik{4}{8} / \im \phik{4}{12}$ has $D^{(8,5)}$ as a composition factor. Calculation shows that in fact it is isomorphic to $D^{(8,5)}$.
\end{example}

%> PsiHomologyIdentified(2,13,4,0); 
%true [
%    <8, [
%        [
%            [
%                [ 8, 5 ]
%            ]
%        ]
%    ], 560>,
%    <4, [], 0>
%]
%true

\subsubsection*{Outline}
In \S\ref{sec:background} below we give some further motivation from simplicial
homology. This section also collects several results on hook-Specht modules and 
discusses earlier related work.
 In~\S\ref{sec:prelim} we give the logical
preliminaries for the proofs of the main theorems. In~\S\ref{sec:epsilonHomology} we prove Theorem~\ref{thm:epsilonHomology}
and in \S\ref{sec:gen} we prove Theorem~\ref{thm:gen}.
The zero homology modules for the two-step boundary maps are instances of
both theorems, but the proofs are independent and involve somewhat different ideas.
In~\S\ref{sec:splitExact} we extend the arguments in \S\ref{sec:gen} to
prove Theorem~\ref{thm:splitExact}.
The final section \S\ref{sec:problems} suggests four directions for future work inspired by Theorems~\ref{thm:epsilonHomology} and~\ref{thm:gen}. 
In particular Conjectures~\ref{conj:odd3} and~\ref{conj:odd5} 
give two attractive binomial identities that 
would be categorified by an extension of these results to odd characteristic.

\section{Background}
\label{sec:background}

\subsection*{Exterior powers of the natural permutation module}

Suppose that~$\F$ has prime characteristic $p$ and let $M = \langle e_1, \ldots, e_n \rangle_\K$ be the
natural permutation module for $\K S_n$.
The $\K S_n$-module $\bigwedge^k\! M$ 
has as an $\K$-basis all \emph{$(k-1)$-simplices} $e_{i_1} \wedge \cdots \wedge e_{i_k}$
where $1 \le i_1 < \cdots < i_k \le n$.
For $k \in \N$, the \emph{boundary map} $\delta_k : \bigwedge^k\! M \rightarrow \bigwedge^{k-1}\! M$ is defined by
\[ (e_{i_1} \wedge \cdots \wedge e_{i_k}) \delta = \sum_{\c=1}^k (-1)^{\c-1} e_{i_1} \wedge
\cdots \wedge \widehat{e_{i_\c}} \wedge \cdots \wedge e_{i_k} \]
where $\widehat{e_{i_\c}}$ indicates that this factor is omitted. A short calculation
shows that $\delta_{k+1} \delta_{k} = 0$, and so $\im \delta_{k+1} \subseteq \ker \delta_{k}$,
for all $k$. Thus
\begin{equation}
\label{eq:deltaComplex} 
\bigwedge^n M \xrightarrow{\delta_n} \bigwedge^{n-1} M \xrightarrow{\delta_{n-1}} \cdots \xrightarrow{\ \delta_3\ } \bigwedge^2 M \xrightarrow{\ \delta_2\ }
M \xrightarrow{\ \delta_1\ } \F \end{equation}
is a chain complex. Given $v \in \ker \delta_k$ a variation on the product rule for derivatives implies that
\begin{equation}\label{eq:deltaSuspension}
(e_1 \wedge v) \delta_{k+1} = v - e_1 \wedge (v \delta_k) = v,\end{equation}
and so~\eqref{eq:deltaComplex} is exact.
Correspondingly, as is very well known, the solid $(n-1)$-simplex has 
zero homology in all non-zero dimensions. (Note that $\bigwedge^k\! M$ corresponds to 
$(k-1)$-simplices,
and so the final map \smash{$M \xrightarrow{\delta_1} \F$} is omitted when computing the geometric homology.)
%taking dimensions we obtain $\sum_{k} (-1)^k \binom{n}{k} = 0$.
The identity~\eqref{eq:deltaSuspension} is the algebraic statement of the \emph{suspension trick} showing 
that an arbitrary cycle $v \in \im \delta_{k+1}$ is a boundary lying in $\ker \delta_k$:
see Figure~1 below. We adapt this trick in 
Lemma~\ref{lemma:suspension}: this lemma is critical to the proof of
Theorem~\ref{thm:epsilonHomology}, and is also used in the proof of Theorem~\ref{thm:gen}(ii). % below.

\newcommand{\arr}[3]{\draw[line width=0.25pt,-{Latex[width=4pt]}] (#1.center)--($(#1)!#3!(#2)$);%
                     \draw[line width=0.25pt] ($(#1)!#3!(#2)$)--(#2.center);}

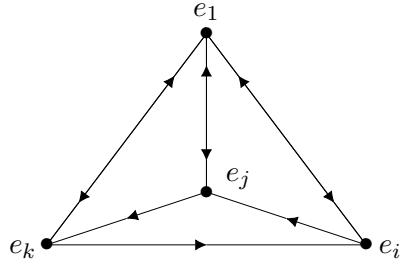
\begin{figure}[h!]
\begin{center}
\begin{tikzpicture}[x = 0.7cm, y = 0.7cm]
\node[label={[xshift=9pt,yshift=-15pt]$e_{i}$}] (A) at (5,0) {$\bullet$};
\node[label={[xshift=12pt,yshift=-9pt]$e_{j}$}] (B) at (2,1) {$\bullet$};
\node[label={[xshift=-9pt,yshift=-15pt]$e_{k}$}]  (C) at (-1,0) {$\bullet$};
%\node (D) at (-0.6,-0.8) {};
%\node (E) at (4.6,-0.8) {};
\node[label={[xshift=0pt,yshift=-5pt]$e_{1}$}] (T) at (2,4) {$\bullet$};
\arr{A}{B}{0.5}
\arr{B}{C}{0.5}
%\arr{C}{D}{0.8}
%\arr{E}{A}{0.4}
\arr{A}{T}{0.8}
\arr{T}{A}{0.8}
\arr{T}{B}{0.8}
\arr{B}{T}{0.8}
\arr{C}{A}{0.5}
\arr{C}{T}{0.8}
\arr{T}{C}{0.8}
%\arr{T}{[xshift=3pt]B}{0.8}
%\arr{[xshift=-3pt]B}{T}{0.7}

%\arr{[xshift=3pt]A}{[xshift=3pt]T}{0.5}
%\draw (E.center)--(A.center);
\end{tikzpicture}
\end{center}
\caption{Suspension trick: the cycle $e_i \wedge e_j + e_j \wedge e_k + e_k \wedge e_i$
is equal to the boundary
$(e_{1} \wedge e_{i} \wedge e_{j}) \delta_3 + (e_{1} \wedge e_j \wedge e_k) \delta_3 + 
(e_1 \wedge e_k \wedge e_i) \delta_3$.}
\end{figure}

Let $U = \langle e_i - e_1 : 1 < i \le n \rangle$. 
Then $U$ is a submodule of $M$ isomorphic to the Specht module $S^{(n-1,1)}$
and $U = \ker \delta_1$.
By~\eqref{eq:deltaSuspension}, it easily follows that
$\bigwedge^k U \subseteq \ker \delta_k$ for each $k$. On the other hand,
since 
\[ (e_{i_1}-e_1) \wedge \cdots \wedge (e_{i_k} - e_1) = 
(e_1 \wedge e_{i_1} \wedge \cdots e_{i_k}) \delta_{k+1} \in \im \delta_{k+1} \]
we have $\bigwedge^k U \supseteq \im \delta_{k+1}$. By exactness we deduce
that $\bigwedge^k U = \ker \delta_k$.
If~$p$ does not divide $n$ then $M = U \oplus \langle e_1 + \cdots + e_n \rangle$ and so
$\bigwedge^k \! M \cong \bigwedge^k U \oplus \bigwedge^{k-1} U \cong \ker \delta_k \oplus
\im \delta_k$ and~\eqref{eq:deltaComplex} is split exact. %, as claimed in the introduction when $p=2$.

To
motivate a key step in the proofs of Theorems~\ref{thm:gen} and Theorem~\ref{thm:splitExact}, we sketch an alternative proof of this decomposition,
related to the suspension trick. For $k \in \N$, define $f_k : \bigwedge^{k-1}\! M
 \rightarrow \bigwedge^k\! M$ by
 $(e_{i_1} \wedge \cdots \wedge e_{i_{k-1}}) f_k
= e_1 \wedge e_{i_1} \wedge \cdots \wedge e_{i_{k-1}}$. Then %, by the suspension trick,
$\delta_k f_k + f_{k+1}\delta_{k+1} = \mathrm{id}$ for each $k$. Hence the maps $f_k$ define
a chain homotopy between~\eqref{eq:deltaComplex} and the zero complex. As it stands,
$f_k$ is not an $\F S_n$-homomorphism, but replacing $f_k$ with
the symmetrized map $F_k$ defined by $(e_{i_1} \wedge \cdots \wedge e_{i_{k-1}}) F_k
= (e_1 + \cdots + e_n) \wedge (e_{i_1} \wedge \cdots \wedge e_{i_{k-1}})$, we
get 
\begin{equation} \label{eq:Phtpy} \delta_k F_k + F_{k+1}\delta_{k+1} = n\, \mathrm{id}. \end{equation}
Since $F_k F_{k+1} = 0$, a basic argument from homotopy theory, which
 we repeat in the proof of Proposition~\ref{prop:splitExact}, shows that if $p$ does not divide $n$ then
 $\bigwedge^k\! M = \im F_k \oplus \im \delta_{k+1}$ for every $k$ and so~\eqref{eq:deltaComplex}
is split exact.

There is a canonical isomorphism 
\begin{equation}\label{eq:hookIso} \ker \delta_k \cong S^{(n-k,1^k)} \end{equation}
first constructed by
Hamernik \cite{Hamernik} 
in  the case $n=p$ and Peel \cite[Proposition~2]{PeelHooks} 
in general. (For the definition of Specht modules and polytabloids 
see \cite[Ch.~4]{James}.) The isomorphism is defined by sending $(e_{i_1} - e_1) \wedge
\cdots \wedge (e_{i_k} - e_1)$ to the polytabloid $e_t$ where~$t$ is the unique standard
$(n-k,1^k)$-tableau
having first column entries $1, i_1, \ldots, i_r$.
By the Standard Basis Theorem (see \cite[Corollary 8.5]{James}), this defines a linear isomorphism.
It follows easily from the definition
of polytabloids that it commutes with the permutations fixing~$1$; a short calculation with Garnir relations 
(see  \cite[Proposition 2.3]{MullerZimmermann} or \cite[Proposition 5.1(b)]{GiannelliLimWildon}) 
shows that it commutes with~$(1,2)$.

The following result completely determines the structure
of $\bigwedge^k\! M$ when~$p$ is odd. It was proved in the author's D.~Phil thesis \cite[\S 1.3]{WildonThesis}
using the ideas in Hamernik \cite{Hamernik}, Peel \cite{PeelHooks} and James \cite[Theorem 24.1]{James}.

\begin{proposition} Let $p$ be odd. We have $\bigwedge^0\! M \cong \K$ and $\bigwedge^n\! M \cong \mathrm{sgn}$.
\begin{thmlistaltlesstop}
\item If $p$ does not divide $n$ and $k \in \{1,\ldots, n-1\}$ then $S^{(n-k,1^k)}$ is simple and 
$\bigwedge^k\! M \cong S^{(n-k,1^k)}
\oplus S^{(n-k-1,1^{k-1})}$ is semisimple.

\item Suppose $p$ divides $n$. Let $D = U / \langle e_1 + \cdots + e_n \rangle$ and let $D_k$ denote
$\bigwedge^k \! D$. Then 
$D_k$ is simple  and there is a non-split exact sequence
$D_{k-1} \hookrightarrow S^{(n-k,1^k)} \twoheadrightarrow D_{k}$ for each $k \in \{1,\ldots, n-2\}$.
For $k \in \{1,\ldots, n-1\}$,
each $\bigwedge^k\! M$ is indecomposable with Loewy layers
\[ \begin{matrix} D_{k-1}\\ D_{k-2} \oplus D_k \\ D_{k-1} \end{matrix}\,, \]
where $D_{-1}$ and $D_{n-1}$ should be ignored when $k=1$ or \hbox{$k=n-1$}.
\end{thmlistaltlesstop}
\end{proposition}

A corollary of this proposition, which may easily be proved directly by 
 considering possible images of the generator $e_1 \wedge \cdots \wedge e_k$
of $\bigwedge^k\! M$, is that~if $p$ is odd and
$|k - \ell| \ge 2$ then $\Hom_{\F S_n}(\bigwedge^k\! M, \bigwedge^\ell\! M) = 0$. 
This rules out a generalization to odd characteristic of the main theorems in which 
$\F \Omega{k}$ is replaced with $\bigwedge^k\! M$. %CORR TO arxiv
At the end of \S\ref{sec:problems} we propose an alternative generalization.

\subsection*{Other related  work}
The maps $\phik{t}{k}$ are critical to James' proof \cite{JamesChar2} of the 
decomposition numbers for Specht modules labelled by two-row partitions. 
(In \cite{JamesChar2}, our map \smash{$\phik{t}{k}$} is denoted $\theta^k_{k-t}$.) James' Lemma~2.7
gives an inductive construction of generators for the module 
\smash{$\bigcap_{t=k-r}^k \ker \phik{t}{k}$};
his Lemma~3.6  shows that the intersection is the same when taken only over those $t$ of the form~$2^\tau$.
James' Lemma 3.5 states that
 $\ker \phik{k}{s+t}$ contains $\ker\phik{k}{s}$ if and only if $\binom{s+t}{s}$ is odd;
we adapt his proof %of the `only if'
to prove the related Proposition~\ref{prop:kernelProper} below.
The example following James' Lemma 2.7 describes some of the submodules in our Example~\ref{ex:ex}.
Later in \cite[Chapter 17, 24]{James}, James revisited these ideas. His Theorem~17.13(i) implies that \smash{$\{2,4,\ldots, 2k\} \sum_{\sigma \in G_\ell} \sigma$}
generates the kernel of \smash{$\phik{k-\ell+1}{k}$} \emph{when} this map is restricted
to the submodule of $\F\Omega{k}$
generated by $\{2,4,\ldots, 2k\} \sum_{\sigma \in G_{\ell-1}} \sigma$. 
(The full kernel is in general larger.)
In particular,
taking $\ell = k-1$ shows that $v_k \in \ker \epsilon_k$. Part of our Theorem~\ref{thm:epsilonHomology}
gives the stronger result that $v_k + \im \epsilon_k$ generates the homology module $\ker \epsilon_k/\im\epsilon_{k+2}$; the proof uses somewhat different ideas to James.
Conjecture~\ref{conj:generation} proposes a generalization of this result.

In \cite{HenkeYoung}, Henke determined the multiplicities of two-row 
Young modules in the two-row Young permutation
modules (isomorphic to the $\F \Omega{k}$) working in arbitrary characteristic. 
In \cite{DotyErdmannHenkeYoung}, Doty, Erdmann and Henke used the Schur algebra in characteristic $2$ to give an explicit
construction of the primitive idempotents in $\End_{\F S_n}(\F\Omega{k})$.
When~\eqref{eq:phiComplex} is split exact, each \smash{$\ker \phik{t}{k}$} is a direct sum of Young modules,
and the projection $\F\Omega{k} \rightarrow \ker\phik{t}{k}$ is the sum of the relevant idempotents.
For instance, in Example~\ref{ex:ex}, $\ker \epsilon_4 \cong Y^{(6)}$ and $\ker \epsilon_2 \cong Y^{(4,2)}$.
In general multiple idempotents are required. For example, take $\tau \in \N_0$,
$t = 2^\tau$, $k = 2^{\tau+1}$ and $n = (3 + 4r) 2^\tau$
with $r \in \N$. By Theorem~\ref{thm:splitExact}, $\ker \phik{t}{k}$
is a direct summand of $\F\Omega{k}$; an argument similar to Example~\ref{ex:splitSurprise}
shows that the trivial module is a proper direct summand of $\ker \phik{t}{k}$.

Earlier, in \cite{MurphyDecomposable}, Murphy proved a number of results on the endomorphism ring of 
$\ker \phik{1}{k} %\gamma_k 
\cong S^{(n-k,1^k)}$ when $p=2$
and used them to determine when this hook-Specht module is decomposable. 
When~$n$ is odd an alternative proof of her criterion can be given using the results in \cite{HenkeYoung},
starting from the observation that $S^{(n-r,1^r)}$ is a direct summand of $\F\Omega{k}$
containing $S^{(n-r,r)}$, and so
is a direct sum of Young modules including $Y^{(n-r,r)}$.

The results on the restricted modules
 $D^{(m+1,m)}\hskip-2pt\res_{S_{2m}}$ and $D^{(m+1,m-1)}\hskip-2pt\res_{S_{2m-1}}$
in Theorem~\ref{thm:epsilonHomology}
were proved by Danz and K{\"u}lshammer in 
\cite[Proposition~3.3]{DanzKulshammer}; the authors' proof uses Kleshchev's very deep
modular branching rule \cite[Theorem 11.2.10]{KleshchevBook}. 
The explicit construction of $D^{(m+1,m-1)}$ in \cite{DanzKulshammer}, attributed to Uno,
also implies these results. The proof here is self-contained
and inductive. %Finally we note that 
The generator
for $D^{(m+1,m)}$ in Theorem~\ref{thm:epsilonHomology}
was first found by Benson (with a different description of the quotient module) in \cite[Lemma~5.4]{BensonSpin}.

Finally we note that there is an extensive theory of resolutions of (dual) Specht modules
by Young permutation modules, beginning with \cite{BoltjeHartmann}; 
the authors' conjectured resolution was proved to be exact in \cite{SantanaYudin} using the Schur algebra. 
Even in the two-row case, the terms in these resolutions are 
sums of multiple Young permutation modules. 
Thus they do not appear to be closely connected to this work.

\section{Preliminary results}\label{sec:prelim}

From now until the final part of \S 7, let $\F$ be a field of characteristic $2$.

\newcommand{\rr}{r}
\subsubsection*{Duality}
Each $\F\Omega{\rr}$ is isomorphic to its dual module $\F\Omega{\rr}^\star$
by a canonical isomorphism sending $X \in \Omega{\rr}$ to the
corresponding element $X^\star$ of the dual basis of $\F \Omega{\rr}^\star$.
Under this identification,
$\phik{t}{\rr} : \F\Omega{\rr} \rightarrow \F\Omega{\rr-t}$ becomes the map 
$\phikstar{t}{\rr} : \F \Omega{\rr-t} \rightarrow \F\Omega{\rr}$ defined~by
\begin{equation}\label{eq:phiStar}
Y \phikstar{t}{\rr} = \sum_{Z \supseteq Y \atop |Z| = |Y| + t\rule{0pt}{5.25pt}} Z \end{equation}

\vspace*{-4pt}
\noindent
for $Y \in \Omega{\rr-t}$. (Note that the domain of $\phikstar{t}{\rr}$ is defined to be
$\F\Omega{\rr-t}$, not $\F\Omega{\rr}$ or $\F\Omega{\rr-t}^\star$.)
This duality explains the symmetry in the inequalities
in~Theorem~\ref{thm:gen}.

\enlargethispage{11pt}

\begin{proposition}{\ } \label{prop:dual}
\begin{itemize}
\item[(i)] For each $r$ there is an isomorphism $\F\Omega{r} \cong \F\Omega{n-r}$.
\item[(ii)] The homology of
\[ \F\Omega{k+t} \xrightarrow{\phik{t}{k+t}}\F\Omega{k} \xrightarrow{\phik{t}{k}}
\F\Omega{k-t} \]
is dual to the homology  of 
\[ \F\Omega{n-k+t}  \xrightarrow{\phik{t}{n-k+t}} \F\Omega{n-k} \xrightarrow{\phik{t}{n-k}}
\F\Omega{n-k-t}. \]
\end{itemize}
\end{proposition}

\begin{proof}
Dualising the first sequence we obtain 
$\F\Omega{k-t} \xrightarrow{\phikstar{t}{k}}\F\Omega{k} \xrightarrow{\phikstarkt}
\F\Omega{k+t}$.
Each $\F\Omega{r}$ is isomorphic to $\F\Omega{n-r}$ by the map 
sending each $Y \in \Omega{r}$ to its complement $\{1,\ldots, n\} \backslash Y \in \Omega{n-r}$.
Applying this isomorphism we obtain the second sequence. In particular, the homology
modules are dual.
\end{proof}

\subsubsection*{Specht modules, Young permutation modules, simple modules}
The Specht module $S^\lambda$ canonically labelled by the partition $\lambda$ of $n$
is defined in \cite[Ch.~4]{James} as a submodule of the Young permutation
module $M^\lambda$. There is a well-known canonical isomorphism
$M^{(n-k,k)} \cong \F\Omega{k}$ defined by sending a tabloid of shape $(n-k,k)$ to the
set of entries in its bottom row. Let
$t$ be the $(n-k,k)$-tableau having $2,4,\ldots, 2k$
in its bottom row. Then the corresponding polytabloid $e_t$ generates $S^{(n-k,k)}$
and
\begin{equation}
\label{eq:twoRowSpecht} e_t \mapsto \{2,4,\ldots, 2k\} \sum_{\sigma \in G_k} \sigma. \end{equation}
%a generating polytabloid for $S^{(n-k,k)}$
%is sent to $\{2,4,\ldots, 2k\} \sum_{\sigma \in G_k} \sigma$.
The simple modules for $\F S_n$ are defined in \cite[Theorem~11.5]{James} as 
the top composition factors of certain Specht modules.
For $2k < n$, let $D^{(n-k,k)}$ denote the simple $\F S_n$-module
canonically labelled by the two-row partition $(n-k,k)$. 
We allow partitions to have zero parts: thus $D^{(n,0)}$ is the trivial $\F S_n$-module.
By \cite[Theorem 11.5]{James} each simple $\F S_n$-module is self-dual.

\vbox{
\begin{lemma}{\ } \label{lemma:twoRowSimples}
\begin{thmlistaltlesstop}
\item If $2k < n$ then $\F\Omega{k}$ has a composition series with factors $D^{(n-r,r)}$ for $r \le k$
in which $D^{(n-k,k)}$ appears exactly once.
\item If $n = 2m$ then $\F\Omega{m}$ has a composition series with factors $D^{(2m-r,r)}$ for $r < m$.
\item If $n = 2m$ then
$D^{(m+1,m-1)}$ is a composition factor of $\F \Omega{k}$ if and only if $k = m-1$, $k=m$ or $k=m+1$.
\item Let $2k < n$ and let $2r < n-1$. 
If $D^{(n-1-r,r)}$ is a composition factor of $D^{(n-k,k)}\!\res_{S_{n-1}}$ then
$k \ge r$.
\end{thmlistaltlesstop}
\end{lemma}}

\vspace*{-12pt}
\begin{proof}
Parts (i) and (ii) are special cases of Theorem 12.1 in \cite{James}. 
Using Proposition~\ref{prop:dual}(i) to reduce to the case $2k \le n$, part (iii) also follows
from this theorem.
The hypothesis for (iv) implies that
$D^{(n-1-r,r)}$ appears in 
\[ \F\Omega{k}\Res_{S_{n-1}} \cong \F\Omegan{k}{n-1} \oplus \F\Omegan{k-1}{n-1},\]
where each bracketed $n-1$ indicates that the summand is a module for $\F S_{n-1}$.
By (i) and (ii) we deduce that $k \ge r$.
\end{proof}

The following consequence of Lemma~\ref{lemma:twoRowSimples} is used in both \S\ref{sec:epsilonHomology} and \S\ref{sec:gen}.

\begin{proposition}\label{prop:cfs}
Let $n \in \N$. 
\begin{thmlistaltlesstop}
\item  If $n = 2m$ then $\F\Omega{m}$ has exactly two composition factors  isomorphic to $D^{(m+1,m-1)}$.
%and $\F\Omega{m+2} \cong \F\Omega{m-2}$ have no such composition factors. 
\item  If $n = 2m+1$ then %the isomorphic modules
$\F\Omega{m}$ and $\F\Omega{m+1}$ are isomorphic and
each has a unique composition factor isomorphic to $D^{(m+1,m)}$.
%and $\F\Omega{m+3} \cong \F\Omega{m-2}$ have no such composition factors.
\end{thmlistaltlesstop}
\end{proposition}

\begin{proof} Recall that $\gamma_k$ denotes $\phik{1}{k}$\hskip-0.5pt.
We use the  one-step sequence
\[ 0\rightarrow \F\Omega{n} \xrightarrow{\gamma_n} \F\Omega{n-1} \xrightarrow{\gamma_{n-1}} \cdots \xrightarrow{\gamma_2} \F\Omega{1} \xrightarrow{\gamma_1} \F\Omega{0} \rightarrow 0. \]
As seen after~\eqref{eq:deltaComplex}, this sequence is exact.
If $n = 2m$ then, by Proposition~\ref{prop:dual}(i) and
Lemma~\ref{lemma:twoRowSimples}(i), the isomorphic modules $\F\Omega{m-1}$ and $\F\Omega{m+1}$ each have $D^{(m+1,m-1)}$ as a composition
factor. By Lemma~\ref{lemma:twoRowSimples}(iii), $D^{(m+1,m-1)}$ is not a composition factor of $\F\Omega{m-2} \cong \F\Omega{m+2}$. Therefore
$D^{(m+1,m-1)}$ must appear twice in $\F\Omega{m}$. The proof is similar when $n = 2m+1$.
\end{proof}

\subsubsection*{Composing multistep maps}
We need a generalization of the result $\phik{t}{k} \phik{t}{k-t} = 0$ proved in the introduction.
Given $s$, $t \in \N_0$, 
we say that the \emph{addition of $s$ to~$t$ is carry free} if $\binom{s+t}{s}$ is odd.
Abusing notation slightly, we may abbreviate this to `$s + t$ is carry free'. 
As motivation, we recall that if $s = \sum_{i=0}^c s_i 2^i$ and $t = \sum_{i=0}^c t_i 2^i$
where $s_i$, $t_i \in \{0,1\}$ for each $i$,
then $s + t$ is carry free if and only if $s_i + t_i \le 1$ for all~$i$, and so $s$ and $t$
can be added in binary without carries.
(This follows immediately from Lucas' Theorem: see for instance \cite[Lemma 22.4]{James}.)

\begin{lemma}\label{lemma:composition}
If $s$, $t \in \N$ then
\[ \phik{s}{k} \phik{t}{k-s} = \begin{cases} \phik{s+t}{k} & \text{if the addition of
$s$ to $t$ is carry free} \\
0 & \text{otherwise.} \end{cases} \]
\end{lemma}

\begin{proof}
The argument in the introduction shows that $\phik{s}{k} \phik{t}{k-s} = \binom{s+t}{s}
\phik{s+t}{k}$. The lemma now follows from the definition of carry free.
\end{proof}

\subsubsection*{Products of sets}
Define the \emph{support} of $v \in \F\Omega{k}$ to be the union of the
$k$-subsets that appear in $v$ with a non-zero coefficient.
The vector space $\bigoplus_{k=0}^n \F\Omega{k}$ becomes a graded algebra with product
defined by bilinear extension of
\[ X \cdot Y = \begin{cases} X \cup Y & \text{if $X \cap Y= \varnothing$} \\
0 & \text{otherwise.} \end{cases} \]
for $X \in \Omega{k}$ and $Y \in \Omega{\ell}$.
We denote this product by concatenation. Except in the warning example
following Lemma~\ref{lemma:split}, we only
take the product of $v \in \F\Omega{k}$ and $w \in \F\Omega{\ell}$ when $v$ and $w$
have disjoint support.

\subsubsection*{The Splitting Rule and the Suspension Lemma}

The product rule for derivatives 
has the following analogue for the multistep boundary maps.

\begin{lemma}[Splitting Rule]\label{lemma:split}
Let $v \in \F\Omega{k}$ and let $w \in \F\Omega{\ell}$. If $v$ and $w$
have disjoint support then
\[ (v w) \phik{t}{k+\ell} = \sum_{s=0}^t (v \phik{s}{k}) (w \phik{t-s}{\ell}). \]
\end{lemma}

\begin{proof}
By bilinearity of the product $\F\Omega{\ell} \times \F\Omega{m} \rightarrow
\F\Omega{k+\ell}$, it suffices to prove the lemma in the special case
when $v$ is an $k$-subset $X$ and $w$ is a disjoint $\ell$-subset $Y$.
It then holds since every $(k + \ell - t)$-subset $Z$ of $X\cup Y$
splits uniquely as a union $(Z \cap X) \cup (Z \cap Y)$ of
a subset of $X$ and a subset of~$Y$. 
\end{proof}

When $t > 1$ the assumption in Lemma~\ref{lemma:split}
that $v$ and $w$ have disjoint support is essential.
For example $(\{1,2\} \{2\}) \epsilon_2 = 0 \epsilon_2 = 0$, but
$(\{1,2\} \epsilon_2) \{2\} + (\{1,2\}\gamma_1) (\{2\}\gamma_1) + \{1,2\} (\{2\} \epsilon_2)
= \varnothing \{2\} + (\{1\} + \{2\})\varnothing  = \{1\}$.

The following lemma is the analogue of~\eqref{eq:deltaSuspension} in \S\ref{sec:background}.

\begin{lemma}[Suspension Lemma]\label{lemma:suspension}
Let $t \in \N$ and let $0 \le \ell < t$. Let $v \in \F\Omega{k}$.
Suppose that $v \in \ker \phio{s}{k}$ whenever $\ell < s \le t$
and that the support of $v$ is disjoint from $X \in \Omega{\ell + t}$.
If the addition of $\ell$ to $t$ is carry free and the
addition of $\ell$ to $t-s$ is not carry free when $0 < s \le \ell$
then
\[ v = \bigl( v (X \phio{\ell}{\ell + t}) \bigr) \phio{t}{k+t}\hskip0.5pt .\]
\end{lemma}

\begin{proof}
By the Splitting Rule the right-hand side is
%and the hypothesis that $v \in \ker \phik{s}{k}$ when $\ell < s \le t$, the right-hand side is 
\begin{equation}
\label{eq:suspensionProof}
\sum_{s=0}^{t} (v \phioo{s}{k}) ( X \phioo{\ell}{\ell+t} \phioo{t-s}{t}). 
\end{equation}
(Here, and in the remainder of the proof, we omit the degrees of the maps to increase readability.)
By hypothesis $v \phi{s} = 0$ if $\ell < s \le t$.
When $0 < s \le \ell$ the addition of $\ell$ to $t -s$
is not carry free, again by hypothesis.
Therefore, by Lemma~\ref{lemma:composition},
we have $X\phi{\ell} \phi{t-s} = 0$ for all such~$s$. 
The only remaining summand in~\eqref{eq:suspensionProof} occurs
when $s=0$, in which case another
application of Lemma~\ref{lemma:composition}
shows that $v (X \phi{\ell}\phi{t}) = v \varnothing
= v$. 
\end{proof}

For example, take $t = 2^\tau$ where $\tau \in \N_0$ and take $k < 2^\tau$. Then $k+ 2^\tau$
is carry free, and if $0 < s \le k$ then $k + (2^\tau -s)$, 
is clearly not carry free, since it has~$2^\tau$ in its binary form. The sets $v = \{n-k+1,\ldots, n\}$ and
$X = \{1,\ldots, k+2^\tau\}$ are disjoint whenever $n-k \ge k+2^\tau$.
Hence the hypotheses of the Suspension Lemma hold provided $n \ge 2k + 2^\tau$ and we get
\[ \{n-k+1,\ldots, n\} = \bigl( \{n-k+1,\ldots, n\} (\{1,\ldots, k+2^\tau\} \phio{k}{k+2^\tau}) \bigr) \phio{2^\tau}{k+2^\tau}.\]
Therefore $\phik{2^\tau}{k+2^\tau} : \F\Omega{k+2^\tau} \rightarrow \F\Omega{k}$ is surjective.
We use a small generalization this argument in the proof of part of Theorem~\ref{thm:gen}(ii).

\section{Two-step homology: proof of Theorem~\ref{thm:epsilonHomology}}\label{sec:epsilonHomology}

Recall that $H_k = \ker \epsilon_{k} / \im \epsilon_{k+2}$.
The outline of the proof is as follows: in Lemmas~\ref{lemma:epsilonGenerators},~\ref{lemma:vkdelta} and~\ref{lemma:epsilonSuspension}
and Proposition~\ref{prop:Hgen} we show that $v_k + \im \epsilon_{k+2}$ generates~$H_k$. 
Using that $v_k$ is supported on a set of size $2k-1$, it follows from the Suspension Lemma
that $H_k = 0$ when $n \ge 2k+2$.
By duality we get the same result when $n \le 2k-2$. 
We then identify the composition factors responsible
for the non-zero homology modules, and  find their structure by induction on $n$.
Thus a large part of the proof
is to show that  $\ker \epsilon_k$ has a generator of `small' support: as motivation 
note that, conversely, if
$\ker \epsilon_k = \im \epsilon_{k+2}$, then
$\ker \epsilon_k$ has a generator supported on $\{1,\ldots, k+2\}$.

Throughout $\gamma_k$ denotes $\phik{1}{k}$ and $\epsilon_k$ denotes~$\phik{2}{k}$\hskip-0.5pt.

\begin{lemma}\label{lemma:epsilonGenerators}
Let $2 \le k \le n-2$. The homology module $H_k$ is generated, as an $\F S_n$-module, by all
$\{n\} v + \{n-1,n\} (v \gamma_{k-1}) + \im \epsilon_{k+2}$ where  $v \in \F\Omega{k-1}$
has support disjoint from $\{n-1,n\}$ and satisfies
$v\epsilon_{k-1} = 0$.
\end{lemma}

\begin{proof}
Given any $X \in \F\Omega{k}$ with support disjoint from $\{n-1,n\}$,
the Splitting Rule implies that 
\[
%\begin{align*} 
X = \bigl( \{n-1,n\} X \bigr) 
\epsilon_{k+2} + \{n-1\} (X\gamma_{k}) + \{n\} (X \gamma_k) + \{n-1,n\} (X \epsilon_k).
\]  %\\
%&\in \{n-1\} (X\gamma_{k}) + \{n\} (X \gamma_k) + \{n-1,n\} (X \epsilon_k)  + \im \epsilon_{k+2}.
%\end{align*}
Since the first summand lies in $\im \epsilon_{k+2}$, and $X$ generates $\F\Omega{k}$
as an $\F S_n$-module, it follows
that $\F \Omega{k} / \hskip-0.75pt\im \epsilon_{k+2}$ is generated by all
$\{n-1\} u + \{n\} v + \{n-1,n\} w + \im \epsilon_{k+2}$ where
$u \in \F\Omega{k-1}$, $v\in \F\Omega{k-1}$ and $w \in \F\Omega{k-2}$ 
have support disjoint from $\{n-1,n\}$. Now, omitting indices on the maps for readability, we have
\[ \begin{split}&(\{n-1\} u +   \{n\} v + \{n-1,n\} w ) \epsilon \\
&\quad = (u \gamma + v \gamma + w) + \{n-1\} (u \epsilon + w \gamma) +
\{ n \} (v \epsilon + w \gamma) + \{n-1,n\} (w \epsilon). \end{split} \]
The right-hand side is zero if and only if 
$u\gamma + v\gamma = w$, $u\epsilon = v\epsilon = w\gamma$ and $w\epsilon = 0$.
The first equation implies that $w \in \im \gamma$, and so $w\gamma = 0$; hence the three equations
are equivalent to $u\gamma + v\gamma = w$ and $u\epsilon = v\epsilon = 0$. Thus 
$H_k$ is generated by all
\[ \{n-1\} u + \{n \} v + \{n-1,n\} (u \gamma + v\gamma) + \im \epsilon_k \]
such that $u\epsilon = v\epsilon = 0$. Applying the transposition $(n-1,n)$ to $\{n \}v + \{n-1,n\}
v\gamma$, we see that $H_k$ is generated by elements of the required form.
\end{proof}

\newcommand{\rk}{k}
\begin{lemma}\label{lemma:vkdelta} %Let $2\rk \le n$.
%\begin{thmlist}
%\item 
If $2 \rk \le n$ then
%We have 
$v_\rk \gamma_k = \{2,4,\ldots, 2(\rk-1)\} \sum_{\tauperm \in G_{\rk-1}} \tauperm$.
%\item $v_\rk + v_k(2\rk-1, 2\rk)$ generates a submodule of $\F \Omega{\rk}$ isomorphic to the
%Specht module $S^{(n-\rk,\rk)}$.
%\end{thmlist}
\end{lemma}

\begin{proof}
Let $w_\rk$ denote the right-hand side. We have
\begin{align*}
v_\rk \gamma_\rk &= \sum_{\tauperm \in G_{\rk-1}} \{2, 4, \ldots, 2(\rk-1) , 2\rk \} \tauperm \gamma_\rk \\
&= \sum_{\tauperm \in G_{\rk-1}} \sum_{j=1}^{\rk-1} \{2 , 4, \ldots, 2(\rk-1) , 2\rk \} \tauperm \,\backslash\, \{(2j) \tauperm\}
+ w_\rk .
\end{align*}
%\begin{align*}
%v_r \gamma &= \sum_{\tauperm \in C_{r-1}} \{2\tauperm, \ldots, 2(r-1) \tauperm, 2r \} \gamma \\
%&= \sum_{\tauperm \in C_{r-1}} \sum_{j=1}^{r-1} \{2 \tauperm, \ldots, 2(r-1) \tauperm, 2r \} \backslash \{2j \tauperm\}
%+ \sum_{w_r}. 
%\end{align*}
For each fixed $j$, the summands for $\tauperm$ and $\tauperm (2j-1,2j)$ are equal, and so cancel.
Therefore $v_\rk \gamma = w_\rk$, as required. 
\end{proof}

\begin{lemma}\label{lemma:epsilonSuspension}
If $v \in \ker \epsilon_k$ has support of size at most $n-3$ then 
$v \in \im \epsilon_{k+2}$.
\end{lemma}

\begin{proof}
By hypothesis, there is a $3$-subset $Z$ of $\{1,\ldots,n\}$ disjoint from the
support of $v$. By the argument seen in the example following the Suspension Lemma
(Lemma~\ref{lemma:suspension}), we have
\[ \bigl( v (Z \gamma_3) \bigr)\epsilon_{k+2} = v. \]
Therefore $v \in \im \epsilon_{k+2}$ as required.
\end{proof}

\begin{proposition}\label{prop:Hgen}
Let $k \in \N_0$. If $2k \le n$ then
$H_k$ is generated by $v_k + \im \epsilon_{k+2}$.
\end{proposition}

\newcommand{\Vm}{V}
\begin{proof}
We work by induction on $n$ dealing with all admissible $k$ at once.
The inductive step below is effective when $k \ge 2$ and $k + 6 \le n$. 
Since $v_0 = \varnothing$ and $v_1 = \{2\}$ generate $\F\Omega{0}$ and $\F\Omega{1}$, respectively,
the result holds if $k < 2$. When $k=2$, Lemma~\ref{lemma:epsilonGenerators}
implies that $H_2$ is generated by all $\{n\} \{j\} + \{n-1,n\} + \im \epsilon_4$,
where $j \in \{1,\ldots, n-2\}$. Therefore $H_2$ is generated by 
$v_2 = \{2,4\} + \{1,4\} + \im \epsilon_4$ as required.
When $k=3$ and $n \in \{6,7,8\}$, or $k=4$ and $n \in \{8,9\}$, or $k=5$ and $n = 10$
the proposition has been checked using the computer algebra package {\sc Magma}.\footnote{{\sc Magma} code for constructing the $\phik{t}{k}$ homomorphisms and verifying
these claims may be downloaded from the author's webpage: \url{www.rhul.ac.uk/~uvah099/.}} 

For the inductive step we may suppose, by the previous paragraph, that $k \ge 2$ and $k+ 6 \le n$.
By Lemma~\ref{lemma:epsilonGenerators}, $H_k$ is generated by the elements $\{n \} v + \{n-1,n\} 
(v\gamma_{k-1})$
for $v \in V$, where \smash{$V = \ker \epsilon^{[n-2]}_{k-1} : \F\Omegan{k-1}{n-2}
\rightarrow %(n-2)}  
\F\Omegan{k-3}{n-2}$}. (The bracketed $n-2$  emphasises that these are modules 
and module homomorphisms for $\F S_{n-2}$.)
The map $\smash{\epsilon^{[n-2]}_{k-1}}$ is part of the  sequence
\[ \F\Omegan{k+1}{n-2} \xrightarrow{\epsilon^{[n-2]}_{k+1}} %[n-2]} 
\F\Omegan{k-1}{n-2}
\xrightarrow{\epsilon^{[n-2]}_{k-1}} %(n-2)}  
\F\Omegan{k-3}{n-2}. \]
Observe that \hbox{$H_{k-1}^{[n-2]} = V / \hskip-0.75pt\im \epsilon^{[n-2]}_{k+1}$}. %(n-2)$}.
Since $2(k-1) \le n-2$, the inductive hypothesis for $n-2$ implies
that \smash{$\Vm / \hskip-0.75pt\im \epsilon^{[n-2]}_{k+1}$} %(n-2)$ 
is generated by 
\smash{$v_{k-1} + \im \epsilon^{[n-2]}_{k+1}$}. %(n-2)$.
Since \smash{$\im \epsilon^{[n-2]}_{k+1}$} is generated by $Y \epsilon_{k+1}$, where
$Y = \{1,\ldots, k+1\}$, it follows that $H_k$ is generated by
$\{n\} v_{k-1} + \{n-1,n\} (v_{k-1} \gamma_{k-1}) + \im \epsilon_{k+2}$
together with $u + \im \epsilon_{k+2}$, where
\[ u = \{n\} (Y \epsilon_{k+1}) + \{n-1,n\} (Y \epsilon_{k+1} \gamma_{k-1}). \]
%\begin{align*}
%&\{n\} v_{k-1} + \{n-1,n\} (v_{k-1} \gamma_{k-1}) + \im \epsilon_{k+2}, \\
%&\{n\} (Y \epsilon_{k+1}) + \{n-1,n\} (Y \epsilon_{k+1} \gamma_{k-1}) 
%+ \im \epsilon_{k+2}.
%\end{align*}
The support of $u$
is $\{1,\ldots,k+1\} \cup \{n-1,n\}$, of size $k+3$. Since $k+6 \le n$, Lemma~\ref{lemma:epsilonSuspension}
implies that $u \in \im \epsilon_{k+2}$.

The first summand in the other generator $\{n\} v_{k-1} + \{n-1,n\} (v_{k-1} \gamma_{k-1})$
is $\sum_{\tauperm \in G_{k-2}} \bigl( \{2,4,\ldots, 2(k-2)\} \tauperm \cup \{ 2(k-1), n\}\bigr)$,
and, by Lemma~\ref{lemma:vkdelta}, %(i), Now just (i)
the second summand is
$\sum_{\tauperm \in G_{k-2}} \bigl( \{2,4, \ldots,  2(k-2) \} \tauperm \cup \{n-1,n\} \bigr)$.
Relabelling so that $n-1$ becomes $2(k-1)-1$ and $n$ becomes $2k$, their
sum becomes~$v_{k}$. Therefore $v_{k} + \im \epsilon_{k+2}$ generates $H_k$.
\end{proof}

\begin{corollary}\label{cor:epsilonZero}
If $2k+2 \le n$ then $H_k = 0$.
\end{corollary}

\begin{proof}
By Proposition~\ref{prop:Hgen}, $H_k$ is generated by $v_k + \im \epsilon_{k+2}$.
The support of $v_k$ is $\{1,\ldots, 2k-2, 2k\}$, of size $2k-1$. Since
$2k+2 \le n$, it follows from Lemma~\ref{lemma:epsilonSuspension} that
$v_k \in \im \epsilon_{k+2}$. Hence $H_k = 0$.
%Since $2k+2 \le n$, there
%exists a $3$-subset $Z$ of $\{1,\ldots, n\}$ disjoint from the support 
%$\{1,\ldots, 2(k-1)\} \cup \{2k\}$ of $v_k$. As in the example
%following the Suspension Lemma (Lemma~\ref{lemma:suspension})
%we have
%\[ \bigl( v_k (Z \gamma_3) \bigr) \epsilon_{k+2} = v_k. \]
%Hence $v_k \in \im \epsilon_{k+2}$ and $H_k = 0$.
\end{proof}

By the duality in Proposition~\ref{prop:dual}(i) we may assume that $2k \le n$. Therefore the previous corollary
determines all the homology modules $H_k$ except when $k = m$ and either $n = 2m$ or $n = 2m+1$.
In these cases the non-zero homology reflects the obstruction to exactness identified
in Proposition~\ref{prop:cfs}.

To complete the proof of Theorem~\ref{thm:epsilonHomology} we show, by induction on $n$,
that the generator $v_m + \im \epsilon_{m+2}$ of $H_m$ given by
Proposition~\ref{prop:Hgen} generates the claimed modules.
The base case is $n=1$, 
in which case
the chain complex $\F\Omega{2} \rightarrow \F\Omega{0} \rightarrow \F\Omega{-2}$
has two zero modules and homology $H_0 = \F\Omega{0} \cong \F \cong D^{(1,0)}$, as required.

\subsubsection*{Inductive step even to odd}
Suppose that $n=2m+1$ so $n-1 = 2m$. 
The restriction of the sequence 
$\F\Omega{m+2} \xrightarrow{\epsilon_{m+2}} \F\Omega{m} \xrightarrow{\epsilon_m} \F\Omega{m-2}$ to $S_{2m}$ is
the direct sum of
\begin{align*} &\F\Omegan{m+2}{2m} \xrightarrow{\epsilon_{m+2}} 
\F\Omegan{m}{2m}\, \xrightarrow{\,\,\,\hskip0.5pt\epsilon_{m}\,\,\,\hskip0.5pt} \F\Omegan{m-2}{2m},
%\label{eq:res1} 
\\
&\F\Omegan{m+1}{2m} \xrightarrow{\epsilon_{m+1}} \F\Omegan{m-1}{2m} 
\xrightarrow{\epsilon_{m-1}} \F\Omegan{m-3}{2m}. %\label{eq:res2}
\end{align*}
(For readability, and since the distinction is no longer so vital, we omit the $[2m]$ label on the 
two-step boundary maps.)
By induction the second sequence is exact. Again by induction, the first
has non-zero homology $E^{(m+1,m-1)}$ in degree $m$. Therefore
\[ H_m \Res_{S_{2m}} 
\cong\, \begin{matrix} D^{(m+1,m-1)} \\ D^{(m+1,m-1)} \end{matrix}\, . \] 
By Lemma~\ref{lemma:twoRowSimples}(iv), the two-row simple modules for $\F S_{2m+1}$ 
whose restrictions to $S_{2m}$ may have $D^{(m+1,m-1)}$ as a composition
factor are $D^{(m+1,m)}$ and $D^{(m+2,m-1)}$. By Proposition~\ref{prop:cfs}(ii), $D^{(m+1,m)}$
appears exactly once in $H_m$. By Nakayama's Conjecture (see \cite[6.1.21]{JK}),
$D^{(m+2,m-1)}$ is in a different block to $D^{(m+1,m)}$. Since
$H_m \!\hskip-0.5pt\res_{S_{2m}}$ is indecomposable, we have $H_m \cong D^{(m+1,m)}$
as required.
%Finally, the submodule of $H_m$ generated by $v_m + v_m (2m-1,2m) + \im \epsilon_{m+2}$ is non-zero,
%since otherwise, by Lemma~\ref{lemma:vkdelta}(ii), the unique top composition factor $D^{(m+1,m)}$
%of $S^{(m+1,m)}$ appears in $\im \epsilon_{m+2}$, and so in $\F\Omega{m+2}$, contradicting
%Lemma~\ref{lemma:twoRowSimples}(i). 

\subsubsection*{Inductive step odd to even}
Suppose that $n=2m$ so $n-1 = 2m-1$. 
The restriction of the sequence 
$\F\Omega{m+2} \rightarrow \F\Omega{m} \rightarrow \F\Omega{m-2}$ to $S_{2m}$ is
the direct sum of
\begin{align*} &\F\Omegan{m+2}{2m-1}\xrightarrow{\epsilon_{m+2}} 
\F\Omegan{m}{2m-1} \xrightarrow{\,\,\,\hskip0.5pt\epsilon_{m}\,\,\,\hskip0.5pt} \F\Omegan{m-1}{2m-1},
%\label{eq:res1} 
\\
&\F\Omegan{m+1}{2m-1} \xrightarrow{\epsilon_{m+1}} \F\Omegan{m-1}{2m-1} \xrightarrow{\epsilon_{m-1}} \F\Omegan{m-3}{2m-1}. %\label{eq:res2}
\end{align*}
By Proposition~\ref{prop:dual} these sequences are dual to one another. By induction,
each has homology $D^{(m,m-1)}$. Hence
\[ H_m \Res_{S_{2m-1}} \cong D^{(m,m-1)} \oplus D^{(m,m-1)}. \]
By Lemma~\ref{lemma:twoRowSimples}(iv), the only two-row simple module for $\F S_{2m}$ whose restriction
to $S_{2m-1}$ may have $D^{(m,m-1)}$
as a composition factor is $D^{(m+1,m-1)}$. By Proposition~\ref{prop:cfs}(i), $D^{(m+1,m-1)}$ 
appears exactly twice in $H_m$.
Hence either $H_m \cong D^{(m+1,m-1)} \oplus D^{(m+1,m-1)}$ or $H_m$ is a non-split extension
of $D^{(m+1,m-1)}$ by itself.
By Proposition~\ref{prop:Hgen}, $H_m$ is generated by $v_m + \im \epsilon_{m+2}$.
Therefore $H_m$ is cyclic. Since the direct sum of two non-zero isomorphic modules 
is not cyclic, it follows that $H_m$ is a non-split extension, as required.

\medskip
This completes the proof of Theorem~\ref{thm:epsilonHomology}. As a corollary we get
a new proof that $\dim D^{(m+1,m-1)} = 2^{m-1}$ and $\dim D^{(m+1,m)} = 2^m$.
For this we need the binomial identity
\begin{equation}\label{eq:binom} \sum_{j} (-1)^j\binom{2m+1}{2j}  
= \begin{cases} (-1)^{m/2} 2^m & \text{if $m$ is even} \\
(-1)^{(m+1)/2} 2^m & \text{if $m$ is odd,}\end{cases}
\end{equation}
which is most easily proved by taking real parts in 
\[ 2^m \i^m + 2^m \i^{m+1} =  (1+\i)^{2m+1} = 
\sum_j (-1)^j \binom{2m+1}{2j} + \i \sum_j (-1)^j \binom{2m+1}{2j+1}.\]

%\begin{lemma}\label{lemma:epsilonBinomialSums} Let $m \in \N_0$. We have
%\begin{align*}
%\sum_{j} (-1)^j\binom{2m}{2j}  &= \begin{cases} (-1)^{m/2} 2^m  & \text{if $m$ is even} \\
%0 & \text{if $m$ is odd.} \end{cases} \\
%\sum_{j} (-1)^j\binom{2m}{2j+1}  &= \begin{cases} 0 & \text{if $m$ is even}  \\
%(-1)^{(m-1)/2}2^m  & \text{if $m$ is odd.} \end{cases} \\
%\sum_{j} (-1)^j\binom{2m+1}{2j}  
%&= \begin{cases} (-1)^{m/2} 2^m & \text{if $m$ is even} \\
%(-1)^{(m+1)/2} 2^m & \text{if $m$ is odd}\end{cases}
%\end{align*}
%and $ \sum_{j} (-1)^j\binom{2m+1}{2j+1} = (-1)^m \sum_{j} (-1)^j\binom{2m+1}{2j}$.
%\end{lemma}
%
%\begin{proof}
%For $m \in \N_0$ we have
%\[ 2^m i^m = (1+i)^{2m} = \sum_j (-1)^j \binom{2m}{2j} + i \sum_j (-1)^j \binom{2m}{2j+1}. \]
%The first and second identities follow by comparing real and imaginary parts. The proof of the
%third identity is similar, taking real parts in 
%\[ 2^m i^m + 2^m i^{m+1} =  (1+i)^{2m+1} = 
%\sum_j (-1)^j \binom{2m+1}{2j} + i \sum_j (-1)^j \binom{2m+1}{2j+1}.\] 
%The fourth follows
%by comparing imaginary parts, or can be deduced from the
%third by replacing $\binom{2m+1}{2j}$ with $\binom{2m+1}{2m+1-2j}$.
%\end{proof}

\begin{corollary}\label{cor:dimensions}
We have $\dim D^{(m+1,m-1)} = 2^{m-1}$ and $\dim D^{(m+1,m)} = 2^m$.
\end{corollary}

\begin{proof}
By part of Theorem~\ref{thm:epsilonHomology}, we have $D^{(m+1,m-1)}\hskip-2.5pt\res_{S_{2m-1}} \cong D^{(m,m-1)}$.
It therefore suffices to prove the second claim. Suppose that $m$ is even. Consider the chain complex
of $\F S_{2m+1}$-modules
\[ \scalebox{0.94}{$\displaystyle 0 \rightarrow \F\Omega{2m} \xrightarrow{\epsilon_{2m}} \cdots \xrightarrow{\epsilon_{m+4}}
\F\Omega{m+2} \xrightarrow{\epsilon_{m+2}} \F\Omega{m} \xrightarrow{\epsilon_{m}} \F\Omega{m-2} 
\xrightarrow{\epsilon_{m-2}} 
\cdots \xrightarrow{\epsilon_{2}} \F\Omega{0} \rightarrow 0.$} \]
%\[ \F\Omega{2m} \rightarrow \cdots \rightarrow
%\F\Omega{m+2} \rightarrow \F\Omega{m} \rightarrow \F\Omega{m-2} \rightarrow 
%\cdots \rightarrow \F\Omega{0}. \]
By Theorem~\ref{thm:epsilonHomology} this chain complex has non-zero homology
 uniquely in degree~$m$, where $H_m \cong D^{(m+1,m)}$. The 
 alternating sum of the dimensions of the modules in a chain complex
  agrees with the alternating sum of the dimensions
 of the homology modules.
Hence %,  %taking the alternating sum of Betti numbers,
\[ \sum_{j=0}^m (-1)^j \dim \F\Omega{2j} = \sum_{j=0}^m (-1)^j \dim H_{2j} = (-1)^{m/2} \dim D^{(m+1,m)}. \]
Since the left-hand side is $\sum_{j=0}^m (-1)^j \binom{2m+1}{2j}$, the result follows
from~\eqref{eq:binom}. The proof is similar if $m$ is odd.
\end{proof}

We end by using the one-step boundary maps $\gamma_k : \F \Omega{k} \rightarrow
\F \Omega{k-1}$ to give
a more explicit description of
the non-split extension in Theorem~\ref{thm:epsilonHomology}.
The following calculation is required.

%In the proof $X \symdiff Y$ denotes the symmetric difference of sets~$X$ and $Y$.

\begin{lemma}\label{lemma:imToKer} If $0 \le k \le n-2$ then
$(\im \epsilon_{k+2}) \gamma_k \gamma_k^\star \subseteq \ker \epsilon_k$.
\end{lemma}

\begin{proof}
Fix $Z \in \Omega{k+2}$. 
If $Y \in \Omega{k}$ 
has a non-zero coefficient in $Z \epsilon_{k+2} \gamma_k \gamma_k^\star$ then either $Y = Z \backslash \{i,i'\}$, for distinct $i,i' \in Z$
or $Y = Z \cup \{ j\} \backslash \{i,i',i''\}$
for distinct $i,i', i'' \in Z$ and $j \not\in Z$. 
In the former case the coefficient of $Y$ is $k$ and in the latter it is $1$.
Therefore $\epsilon_{k+2} \gamma_k \gamma_k^\star = k \epsilon_{k+2} + \psi$
where 
\[ Z \psi = \sum_{i,i',i'' \in Z\atop j \not\in Z} 
\bigl( Z \cup \{j\}
\backslash \{i,i',i''\} \bigr). \]
%where the sum is over distinct $i,i', i'' \in X$ and $j \not\in X$. 
Since $\epsilon_{k+2} \epsilon_k = 0$, it suffices 
to prove that $\psi \epsilon_k= 0$. We may suppose that $k \ge 2$. 
If
$X \in \Omega{k-2}$ has a non-zero coefficient in $Z \psi \epsilon_k$
then either $X = Z \backslash D$ where $D \subseteq Z$ and $|D| = 4$ or
$X = Z \cup \{j\} \backslash E$ where $E \subseteq Z$, $|E| = 5$ and $j \not\in Z$.
%, $j \not\in E$. 
In both cases the coefficient is in fact zero: 
in the first there are $\binom{4}{3}$ choices for $\{i,i',i''\} \subseteq D$
and in the second there are $\binom{5}{3}$ choices for $\{i,i',i''\} \subseteq E$.
\end{proof}

Let $n = 2m$ be even 
and let $U$ be the submodule of $\F\Omega{m}$ generated
by $v_m + v_m(2m-1,2m)$.

\begin{proposition}\label{prop:chain}
Under the canonical isomorphism $\F\Omega{m} \cong
M^{(m,m)}$, the image of $U$ is $S^{(m,m)}$. There is a chain
\[   \rad U + \im \epsilon_{m+2}\subseteq U + \im \epsilon_{m+2} \subseteq
\ker \epsilon_m \]
in which the two quotients are isomorphic to $D^{(m+1,m-1)}$. 
\end{proposition}

\begin{proof}
By Theorem~\ref{thm:epsilonHomology}, 
$v_m \in \ker \epsilon_{m}$. Therefore $U$ is a submodule
of $\ker \epsilon_m$. By~\eqref{eq:twoRowSpecht} in \S\ref{sec:prelim},
under the  canonical isomorphism $\F\Omega{m} \cong M^{(m,m)}$, 
the image of $v_m + v_m (2m-1,2m)$ is
the polytabloid $e_t$, where $t$ is the standard tableau of shape $(m,m)$
having $\{2,4,\ldots, 2m\}$ in its bottom row; 
this polytabloid generates the Specht module 
$S^{(m,m)}$. Therefore $U \cong S^{(m,m)}$.

By the Branching Rule (see \cite[Theorem 9.3]{James})
the restriction of
$S^{(m,m)}$ to $S_{2m-1}$ is $S^{(m,m-1)}$; this module has $D^{(m,m-1)}$
as its unique top composition factor. By Lemma~\ref{lemma:twoRowSimples}(iv),
the only two-row simple module for $\F S_{2m}$ whose restriction to
$S_{2m-1}$ may have $D^{(m,m-1)}$ as a composition factor is $D^{(m+1,m-1)}$.
Therefore, as noted by Benson in \cite[Lemma 5.2]{Benson}, $S^{(m,m)}$
has $D^{(m+1,m-1)}$ as its unique top composition factor, and the multiplicity
of~$D^{(m+1,m-1)}$ in $S^{(m,m)}$ is $1$. Hence $U / \rad U \cong D^{(m+1,m-1)}$.
By Lemma~\ref{lemma:twoRowSimples}(iii), 
$D^{(m+1,m-1)}$ is not a composition factor of $\im \epsilon_{m+2}$.
Since $\ker \epsilon_m / \im \epsilon_{m+2}$ has two composition factors of $D^{(m+1,m-1)}$,
it follows that
 the chain has the claimed quotients.
\end{proof}

\begin{proposition}\label{prop:hom}
Let $n = 2m$ be even.
The endomorphism $\gamma_m \gamma_m^\star$ of $\F\Omega{k}$ restricts
to an endomorphism of $\ker \epsilon_{m}$ satisfying
\begin{thmlist}
\item $v_m \gamma_m \gamma_m^\star = v_m + v_m(2m-1,2m)$;
\item $U \gamma_m \gamma_m^\star = 0$;
\item $(\im \epsilon_{m+2}) \gamma_m \gamma_m^\star \subseteq \im \epsilon_{m+2}$.
\end{thmlist}
\end{proposition}

\begin{proof}
By Lemma~\ref{lemma:vkdelta},
$v_m \gamma_m = \{2,4,\ldots, 2(m-1)\} \sum_{\tauperm \in G_{m-1}}
\tauperm$. Hence
\[ v_m \gamma_m \gamma_m^\star = 
\sum_{\tauperm \in G_{m-1}} \sum_{1 \le i \le 2m \atop
i \not \in \{2,4,\ldots, 2(m-1)\}\tauperm } \bigl( \{2,4,\ldots, 2(m-1)\} \cup \{i\} \bigr). \]
There are summands corresponding to the pairs $(\tauperm,2j)$ and $(\tauperm(2j-1,2j),2j-1)$
if and only if $(2j)\tauperm = 2j-1$; when present, these summands are equal
are so cancel. The summands for $i = 2m$ give $v_m$ and the summands for $i=2m-1$
give $v_m (2m-1,2m)$. Hence
$v_m \gamma_m \gamma_m^\star = v_m + v_m (2m-1,2m)$, proving (i). 
Moreover, since $\bigl( 1 + (2m-1,2m)\bigr)^2 = 0$,
we have $\bigl(v_m + v_m(2m-1, 2m)\bigr) \gamma_m \gamma_m^\star = 0$. Hence
$U \gamma_m \gamma_m^\star = 0$, proving (ii).

By Lemma~\ref{lemma:imToKer}, $(\im \epsilon_{m+2}) \gamma_m \gamma_m^\star
\subseteq \ker \epsilon_{m+2}$. 
By Lemma~\ref{lemma:twoRowSimples}(iii), $\im \epsilon_{m+2}$ does not have $D^{(m+1,m-1)}$
as a composition factor. It therefore follows from Proposition~\ref{prop:chain}
and the Jordan--H{\"o}lder Theorem
that $(\im \epsilon_{m+2}) \gamma_m \gamma_{m}^\star \subseteq \im \epsilon_{m+2}$
as required for~(iii).
\end{proof}

\begin{corollary}\label{cor:nsplit}
Let $n = 2m$. The map $\theta : H_m \rightarrow H_m$ induced by restricting
$\gamma_m \gamma_m^\star$ to $\ker \theta_m$ is a well-defined $\F S_n$-endomorphism
of $H_m$ such that $\theta \not=0$ and $\theta^2 = 0$.
\end{corollary}

\begin{proof}
By Proposition~\ref{prop:hom}, $\theta$ is well-defined. By Theorem~\ref{thm:epsilonHomology},
$H_m$ is 
generated by $v_m + \im \epsilon_{m+2}$. Therefore
$H_m \theta$ is generated by $v_m + v_m(2m-1,2m) + \im \epsilon_{m+2}$; 
by Propositions~\ref{prop:chain} and~\ref{prop:hom}(ii) this is a non-zero element of $H_m$ lying in $\ker \theta$.
\end{proof}

%Consider $\im \epsilon_{m+2}$. By Lemma~\ref{lemma:imToKer}, 
%we have $(\im \epsilon_{m+2})\gamma_k\gamma_k^\star \subseteq \ker \epsilon_m$.
%Since $\ker \epsilon_{m} / (\rad U + \im \epsilon_{m+2}) 
%
%
%By Proposition~\ref{prop:dual}(i), $\F\Omega{m-2} \cong \F\Omega{m+2}$
%and by Lemma~\ref{lemma:twoRowSimples}(ii),
%$D^{(m+1,m-1)}$ does not appear in $\F \Omega{m-2}$. Therefore
%$v_m + v_m(2m-1,2m) + \im \epsilon_{m+2}$ is a generator for
%$S^{(m,m)} / S^{(m,m)} \cap \im \epsilon_{m+2} \cong D^{(m+1,m-1)}$.

\section{Proof of Theorem~\ref{thm:gen}}\label{sec:gen}

In this section we prove the characterization in Theorem~\ref{thm:gen} of when

\vspace*{-15pt}
\[ \tag{\ref{eq:seq}}
\F\Omega{k+t}\xrightarrow{\phik{t}{k+t}}\F\Omega{k}\xrightarrow{\phik{t}{k}}\F\Omega{k-t}\] 
is exact. %The proof of Theorem~\ref{thm:gen} is completed in \S\ref{sec:splitexact}.
We showed in \S\ref{sec:background} that~\eqref{eq:seq} is always exact when $t=1$.
Thus Theorem~\ref{thm:gen}(i) is a sufficient condition.
Clearly~\eqref{eq:seq} is not exact when both $k+t > n$ and $k -t < 0$ and so
only the middle module is non-zero.
 In
\S\ref{subsec:injsurj} we deal with the case when 
there is exactly one zero module.
This leaves the most interesting case of three non-zero modules, described by 
(i) and~(iii). We 
show these conditions are necessary in \S\ref{subsec:nec} and 
sufficient in~\S\ref{subsec:suff}.

The following lemma indicates the obstruction to exactness removed by the condition $k+t \le n-k$.

\begin{lemma}\label{lemma:cfobstruction}
Suppose that $t > 1$ and $k \le n-k < k+t$. Then $\F\Omega{k}$ has a composition factor  not
present in either $\F\Omega{k+t}$ or $\F\Omega{k-t}$.
\end{lemma}

\begin{proof}
By Proposition~\ref{prop:dual}(i) we have $\F\Omega{k+t} \cong \F\Omega{n-(k+t)}$.
By hypothesis, $n-(k+t) < k$.
If $2k < n$ then Lemma~\ref{lemma:twoRowSimples}(i) implies that \smash{$D^{(n-k,k)}$} is a composition
factor of $\F\Omega{k}$ not present in either $\F\Omega{n-(k+t)}$ or $\F\Omega{k-t}$.
In the remaining case $2k = n$ and $\F\Omega{k+t} \cong \F\Omega{k-t}$. 
Since $k-t< k-1$, Lemma~\ref{lemma:twoRowSimples}(iii)
implies that $D^{(k+1,k-1)}$ is a 
composition factor of $\F\Omega{k}$ not present in $\F\Omega{k-t}$.
\end{proof}

\subsection{Surjective and injective maps: Theorem~\ref{thm:gen}(ii)}\label{subsec:injsurj}

There is exactly one zero module in~\eqref{eq:seq} if and only if $k < t \le n-k$ or
$n-k < t \le k$. By Proposition~\ref{prop:dual}(i) we
can reduce to the first case, when the sequence is 
\[ \F\Omega{k+t}\xrightarrow{\phik{t}{k+t}}\F\Omega{k}
\xrightarrow{\phantom{hh}} 0.\] 
It then suffices to prove the following proposition.

\begin{proposition}\label{prop:surj}
Let $k < t \le n-k$ and let $2^\tau$ be the least two-power appearing in the binary form of $t$. 
Then $\phik{t}{k+t} : \F\Omega{k+t} \rightarrow \F\Omega{k}$ is surjective
if and only if $k < 2^\tau$ and $k+t \le n-k$.
\end{proposition}

\begin{proof}
Suppose that $k+t > n-k$. Then, by Lemma~\ref{lemma:cfobstruction},
$\F\Omega{k}$ has a composition factor $D^{(n-k,k)}$ not present in $\F\Omega{k+t}$, and so
\smash{$\phik{t}{k+t}$} is not surjective.
Suppose that $k \ge 2^\tau$.
Since the addition of $2^\tau$ to $t-2^\tau$ is
carry free, Lemma~\ref{lemma:composition} implies that
$\phik{t}{k+t}$ factorizes as $\phik{t-2^\tau}{k+t} \phik{2^\tau}{k+2^\tau}$. 
In the sequence
\[ \F\Omega{k+2^\tau}\xrightarrow{\phik{2^\tau}{k+2^\tau}} \F\Omega{k} \xrightarrow{\phik{2^\tau}{k}}
\F\Omega{k-2^\tau}\]
the map $\phik{2^\tau}{k}$ is non-zero. Since $\im \phik{2^\tau}{k+2^\tau} \subseteq \ker \phik{2^\tau}{k}$,
it follows that $\phik{2^\tau}{k+2^\tau}$ is not surjective. 
Therefore \smash{$\phik{t}{k+t}$} is not surjective.

Conversely, suppose that $k+t \le n-k$ and $k < 2^\tau$. 
Generalizing the example following
the Suspension Lemma (Lemma~\ref{lemma:suspension}), take
$\ell = k$, $v = \{n-k+1,\ldots, n\} \in \ker \phik{t}{k}$ and $X = \{1,\ldots, k + t\}$.
By hypothesis these sets are disjoint. The least two-power
appearing in the binary form of $t$ is $2^\tau$, hence $k+t$ is carry free.
Moreover
if $0 < s \le k$ then 
$k + (t - s)$ is  not carry free, since it has $2^\tau$
in its binary form while $t-s$ does not. Hence
\[ \{n-k+1,\ldots, n\} = \bigl( \{n-k+1,\ldots, n\} (\{1,\ldots,k+t\} \phik{k}{k+t}) \bigr) \phik{t}{k+t} \]
where the left-hand side generates $\F\Omega{k}$. Therefore $\phik{t}{k+t}$ is surjective.
\end{proof}

\subsection{Necessity: Theorem~\ref{thm:gen}(iii)}\label{subsec:nec}

We now suppose that the sequence~\eqref{eq:seq} has three non-zero modules and that $t > 1$ 
and show that the condition in (iii) is necessary for it to be exact.

By Proposition~\ref{prop:dual} we may assume that $2k \le n$. 
Suppose that $n < 2k + t$. Then $k \le n-k < k +t$, so by Lemma~\ref{lemma:cfobstruction}, $\F\Omega{k}$
has a composition factor not present in $\F\Omega{k+t}$ or $\F\Omega{k-t}$.
Therefore~\eqref{eq:seq} is not exact.

It remains to show that if $t$ is not a two-power then~\eqref{eq:seq} is not exact.
The proof of the following proposition uses the
same idea as Lemma~3.5 in \cite{JamesChar2}.

%\begin{lemma}[James, \protect{\cite[page 289]{JamesChar2}}]\label{lemma:James}
%Let $r \in \N$. For any $k \in \N$ with $k \ge r$ and $\ell \in \N_0$ we have
%\[ \Bigl( \sum_{X \subseteq \{1,\ldots, k+\ell\} \atop |X| = k\rule{0pt}{5.25pt}} X \Bigr)
%\phik{r}{k} = \binom{r+\ell}{r} \sum_{Y \subseteq \{1,\ldots, k+\ell\} \atop |Y| = k-r\rule{0pt}{5.25pt}} Y.\]
%%\subseteq \{1,\ldots, k+\ell\} \atop |X| = k} X \phi{t} = \binom{\ell+t}{t} 
%%\sum_{Y \subseteq \{1,\ldots, k+\ell\} \atop |Y| = k-t} Y
%\end{lemma}
%
%\begin{proof}
%Given a $(k-r)$-subset $Y$ of $\{1,\ldots, k+\ell\}$ 
%we may extend $Y$ to an $r$-subset of $\{1,\ldots, k+\ell\}$
%by taking its union with any $r$ elements of the $(r + \ell)$-set
%$\{1,\ldots, k+\ell\} \backslash Y$. Hence the coefficient of each such $Y$
%is $\binom{r+\ell}{r}$.
%\end{proof}

\begin{proposition}\label{prop:kernelProper}
Suppose that $t > s$ and that the addition of $s$ to $t$ is carry free. %is odd. 
If $k \ge s$ then \smash{$\ker \phik{t}{k}$} properly contains \smash{$\ker \phik{s}{k}$}.
\end{proposition}

\begin{proof}
Since $s+t$ is carry free, Lemma~\ref{lemma:composition} implies that
$\phik{t}{k} = \phik{s}{k}\phik{t-s}{k-s}$.
Therefore \smash{$\ker \phik{t}{k}$} contains \smash{$\ker \phik{s}{k}$}.
Since $t> s$, there exists
$\beta$ such that~$2^\beta$ appears in the binary form of $t$ but not in the binary form of $s$. 
Let \smash{$v = \{1,\ldots, k+2^\beta\} \phik{2^\beta}{k+2^\beta}$}. 
Since $t + 2^\beta$ is not carry free,
while $s + 2^\beta$ is carry free, Lemma~\ref{lemma:composition} implies
that $v \phik{t}{k} = 0$ and $v \phik{s}{k} \not=0$.
\end{proof}

\begin{corollary}
Suppose that $t$ is not a two-power. Then~\eqref{eq:seq} is not exact.
\end{corollary}

\begin{proof}
Choose $2^\beta$ such that $2^\beta$ appears in the binary form of $t$ and set $s = t- 2^\beta$.
By Lemma~\ref{lemma:composition} we have \smash{$\phik{t}{k} = \phik{s}{k}\phik{2^\beta}{k-s}$}
and \smash{$\phik{t}{k+t} = \phik{2^\beta}{k+t}\phik{s}{k+s}$}.
Hence
\[ \ker \phik{t}{k} \supseteq \ker \phik{s}{k} \supseteq \im \phik{s}{k+s} \supseteq \im \phik{t}{k+t}\]
where the first containment is strict by Proposition~\ref{prop:kernelProper}. Hence~\eqref{eq:seq}
is not exact.
\end{proof}

%\begin{lemma}[James, \protect{\cite[page 289]{JamesChar2}}]\label{lemma:James}
%Let $r \in \N$. For any $k \in \N$ with $k \ge r$ and $\ell \in \N_0$ we have
%\[ \Bigl( \sum_{X \in \Omegan{k}{k+\ell}} X \Bigr)
%\phi{r} = \binom{r+\ell}{r} \sum_{Y \in \Omegan{k-r}{k+\ell}} Y.\]
%%\subseteq \{1,\ldots, k+\ell\} \atop |X| = k} X \phi{t} = \binom{\ell+t}{t} 
%%\sum_{Y \subseteq \{1,\ldots, k+\ell\} \atop |Y| = k-t} Y
%\end{lemma}

\subsection{Sufficiency: Theorem~\ref{thm:gen}(iii)}\label{subsec:suff}

By Proposition~\ref{prop:dual} we may assume that $2k \le n$. Thus (iii)
holds if and only if $n \ge 2k + t$ and $t = 2^\tau$ is a two-power. We shall show
by induction on $n$ that this condition implies that~\eqref{eq:seq} is exact.
Perhaps surprisingly,
most of the work comes in the base case when $n = 2k + t$, where we prove in Proposition~\ref{prop:splitExact}
the stronger result that~\eqref{eq:seq} is split exact, that is, $\F\Omega{k} = \ker \phik{t}{k} \oplus
C_k$ for an $\F S_n$-module~$C_k$. In this case~\eqref{eq:seq} is part of the chain complex
%where, necessarily,  the chain
%complex with modules in every degree $k + ct$ for $c \in \Z$, namely
\begin{equation}
\label{eq:kseq} \cdots \xrightarrow{\phik{t}{k+3t}} \F\Omega{k+2t} \xrightarrow{\phik{t}{k+2t}} \F\Omega{k+t} \xrightarrow{\phik{t}{k+t}} \F\Omega{k}  
\xrightarrow{\phik{t}{k}} \F\Omega{k-t} \xrightarrow{\phik{t}{k-t}} \cdots. \end{equation}
%is split exact.
Since $n=2k+t$, this chain complex is invariant under the duality in Proposition~\ref{prop:dual};
the case $n=6$, $t=2$ and $k=2$ can be seen in Example~\ref{ex:ex}.

\subsubsection*{Splitting of~\eqref{eq:kseq}}
Motivated by~\eqref{eq:Phtpy} in \S\ref{sec:background}, we 
show that the dual maps~$\phikstar{t}{r}$ defined in~\eqref{eq:phiStar} at the start of \S\ref{sec:prelim}
 define a chain homotopy between~\eqref{eq:kseq} 
and the zero chain complex. The first of the two lemmas below 
 can also be deduced from (2.9) and (2.10) in \cite{MurphyDecomposable}. 
In it $X \symdiff Y$ denotes the symmetric difference of sets~$X$ and $Y$.

\begin{lemma}\label{lemma:phiphistar}
If $Y \in \Omega{k}$ then
\begin{align*}
 Y \phik{t}{k} \phikstar{t}{k} &= \sum_{d=0}^t \binom{k-d}{t-d} 
\sum_{X \in \Omega{k} \atop |X \ssymdiff Y| = 2d\rule{0pt}{5.25pt}} X, \\
 Y \phikstarkt \phik{t}{k+t} &= \sum_{d=0}^t \binom{n-k-d}{t-d} 
\sum_{X \in \Omega{k} \atop |X \ssymdiff Y| = 2d\rule{0pt}{5.25pt}} X.
\end{align*}
\end{lemma}

\begin{proof}
If $X \in \Omega{k}$ is a summand of $Y \phik{t}{k} \phikstar{t}{k}$ then
$X = (Y \backslash D) \cup A$ for unique sets $D \subseteq Y$ and $A \subseteq \{1,\ldots,n\}
\backslash Y$. Clearly $|D| = |A|$. If their common size is $d$ then $|X \symdiff Y| = 2d$.
If $R$ is a $t$-subset of $Y$
%(the elements removed from $Y$ by $\phik{k}{t}$)
such that $R \supseteq D$, we may obtain~$X$ by removing $R$ from $Y$ % in $\phik{t}{k}$
and then inserting the elements of $A \cup (R \backslash D)$. Therefore the coefficient
of $X$ is the number of choices for $R$, namely $\binom{k-d}{t-d}$. 
The proof for $Y \phikstarkt  \phik{t}{k+t}$ is similar.
\end{proof}

\begin{lemma}\label{lemma:binomeqv}
Let $\tau \in \N_0$. The following are equivalent
\begin{thmlistalt}
\item $\binom{k-d}{2^\tau-d} + \binom{n-k-d}{2^\tau-d} \equiv 0$ \emph{mod} $2$ for $1 \le d \le 2^\tau$; \\[-11pt]
%\item $\binom{k-2^\tau+e}{e} + \binom{n-k-2^\tau+e}{e} = 0$ for $0 \le e < 2^\tau$;
\item $\binom{k+e}{e} + \binom{n-k+e}{e} \equiv 0$ \emph{mod} $2$ for $0 \le e < 2^\tau$;\\[-11pt]
\item $\binom{k+2^\sigmapower}{2^\sigmapower} + \binom{n-k+2^\sigmapower}{2^\sigmapower} \equiv 0$ \emph{mod} $2$ 
for $0 \le \sigmapower < \tau$;\\[-11pt]
\item $n \equiv 2k$ mod $2^\tau$.
\end{thmlistalt}
\end{lemma}

\begin{proof}
Observe that if $\ell < 2^\tau$ and $k \equiv k'$ mod $2^\tau$
then 
\[ \text{$k+\ell$ is carry free} \iff \text{$k' + \ell$ is carry free}. \tag{$\dagger$} \]
Replacing $d$ with $2^\tau - e$ in~(i) shows that~(i) is equivalent to
$\binom{k-2^\tau+e}{e} + \binom{n-k-2^\tau+e}{e} \equiv 0$ mod $2$ for $0 \le e < 2^\tau$.
From ($\dagger$) we see that $(k-2^\tau) + e$ is carry free
if and only if $k + e$ is carry free. Therefore (i) is equivalent to (ii). Clearly (ii) implies (iii).
We show that (iii) implies (iv) by induction on $\tau$. If $\tau = 0$ then (iii) is vacuous
and (iv) obviously holds. Suppose that (iii) holds as stated, so
by induction $n \equiv 2k$ mod $2^{\tau}$.
Either $n-k \equiv k$ mod $2^{\tau + 1}$, in which case
$(\dagger)$ implies that
$\binom{k+2^\tau}{2^\tau} \equiv \binom{n-k+2^\tau}{2^\tau}$ mod $2$, or
$n-k \equiv k + 2^\tau$ mod $2^{\tau + 1}$ and similarly $(\dagger)$ implies that
%exactly one of $k+2^\tau$ and $(n-k)+2^\tau$ is carry free, and so
$\binom{k+2^\tau}{2^\tau} + \binom{n-k+2^\tau}{2^\tau} \equiv 1$ mod $2$. This completes the inductive step.
Finally if (iv) holds then $k-d \equiv n-k-d$ mod $2^\tau$ for all $d \in \N$.
By ($\dagger$) this implies (i).
\end{proof}

\begin{lemma}\label{lemma:binomCond}
Let $\tau \in \N_0$. We have
\[ \binom{k-d}{2^\tau-d} + \binom{n-k-d}{2^\tau-d} \equiv 0 \text{ \emph{mod} $2$ } 
 \text{ for $1 \le d \le 2^\tau$} \]
and  $\binom{k}{2^\tau} + \binom{n-k}{2^\tau} \equiv 1$ \emph{mod} $2$ if and only if $n \equiv 2k + 2^\tau$ mod $2^{\tau+1}$.
\end{lemma}

\begin{proof}
By Lemma~\ref{lemma:binomeqv}, the first condition holds if and only if $n \equiv 2k$ mod~$2^\tau$. 
As in the proof of this lemma, the second condition then holds if and only if
exactly one of $k + 2^\tau$ and $(n-k) +2^\tau$ is carry free; equivalently
$n \equiv 2k + 2^\tau$ mod $2^{\tau+1}$.
\end{proof}

\begin{proposition}\label{prop:splitExact}
If $t = 2^\tau$ and
$n \equiv 2k + t$ mod $2^{\tau + 1}$ then
$\ker \phik{t}{k} = \im \phik{t}{k+t}$ and
\smash{$\F\Omega{k} = \ker \phik{t}{k} \oplus \im \phikstar{t}{k}$}.
\end{proposition}

\begin{proof}
By Lemmas~\ref{lemma:phiphistar} and~\ref{lemma:binomCond},
\begin{equation} \label{eq:htpy} \phik{t}{k} \phikstar{t}{k} + \phikstarkt \phik{t}{k+t} = \mathrm{id}. 
\end{equation}
%Applying the left-hand side to an arbitrary $v \in \F\Omega{k}$ 
%it follows that 
Hence, repeating part of a basic argument from homotopy theory, we have \smash{$\F\Omega{k} = \im \phikstar{t}{k} + \im \phik{t}{k+t}$}. 
 If \smash{$v \in \im \phikstar{t}{k} \cap \ker \phik{t}{k}$}
then %since $\phik{t}{k+t}\phik{t}{k} = 0$, we have 
$v\phik{t}{k} = 0$ and, since
$\phikstar{t}{k} \phikstarkt = 0$, we also have $v \phikstarkt = 0$. Evaluating~\eqref{eq:htpy}
at $v$ implies that $v = 0$. 
%If \smash{$v \in \im \phikstar{t}{k} \cap \im \phik{t}{k+t}$}
%then since $\phik{t}{k+t}\phik{t}{k} = 0$, we have $v\phik{t}{k} = 0$ and, since
%$\phikstar{t}{k} \phikstarkt = 0$, we also have $v \phikstarkt = 0$. Evaluating~\eqref{eq:htpy}
%at $v$ implies that $v = 0$. 
Since \smash{$\im \phik{t}{k+t}
\subseteq \ker\phik{t}{k}$} it follows that
\smash{$\F\Omega{k} = \im \phikstar{t}{k} \oplus \ker \phik{t}{k}$} and
$\im \phik{t}{k+t} = \ker \phik{t}{k}$, as required.
\end{proof}

We are now ready to show that Theorem~\ref{thm:gen}(iii) is a sufficient
condition for~\eqref{eq:seq} to be exact.

\begin{proposition}
Let $t$ be a two-power. If $n \ge 2k + t$ then~\eqref{eq:seq} is exact.
\end{proposition}

%\enlargethispage{6pt}
\begin{proof}
We work by induction on $n$ dealing with all admissible $k$ at once. If $n = 2k + t$
then Proposition~\ref{prop:splitExact} shows that~\eqref{eq:seq} is split exact.
Now suppose that $n > 2k + t$ and, inductively, that the sequence of $\F S_{n-1}$-modules
\[ \F\Omegan{k+t}{n-1}\xrightarrow{\phikn{t}{k+t}{n-1}}\F\Omegan{k}{n-1}\xrightarrow{\phikn{t}{k}{n-1}}
\F\Omegan{k-t}{n-1}\] 
is exact. (As usual the bracketed $n-1$  indicates that these are modules,
and importantly, module homomorphisms, for $\F S_{n-1}$.)
Using the product operation on sets defined in \S\ref{sec:prelim},
each element of \smash{$\F\Omega{k}$}
has a unique expression in the form
$U + u \{n\}$ where \smash{$U \in \F\Omegan{k}{n-1}$} and \smash{$u \in \F\Omegan{k-1}{n-1}$}.
Suppose that \smash{$U + u \{n\} \in \ker \phik{t}{k}$}.
By the Splitting Rule (Lemma~\ref{lemma:split}),
\begin{equation}
\label{eq:UuImage} (U+u \{n\}) \phik{t}{k} = U \phik{t}{k} + u \phik{t-1}{k-1} + u \phik{t}{k-1} \{n\}. 
\end{equation}
Hence  \smash{$U \phik{t}{k} + u \phik{t-1}{k-1} = 0$} and $u \phik{t}{k-1} = 0$. Since 
$u \in \F\Omegan{k-1}{n-1}$ and  $n-1 \ge 2(k-1) + t$,
applying the inductive hypothesis
to 
\[ \phikn{t}{k-1}{n-1}: \Omegan{k-1}{n-1} \longrightarrow \Omegan{k-1-t}{n-1} \] 
gives
%\Omegan{k-1}{n-1} \xrightarrow{\phik{t}{k-1}} \Omega{k-1-t}{n-1}$ implies that
\begin{equation}\label{eq:u}
u = v \phikn{t}{k-1+t}{n-1}
\end{equation} 
for some $v \in \F\Omegan{k-1+t}{n-1}$. Substituting~\eqref{eq:u} into~$U \phik{t}{k} + u \phik{t-1}{k-1} = 0$ we obtain
\[ U \phik{t}{k} + v \phik{t}{k-1+t} \phik{t-1}{k-1} = 0.\]
Since $t + (t-1)$ is carry free,
Lemma~\ref{lemma:composition} implies that $\phik{t}{k-1+t}\phik{t-1}{k-1} = \phik{t-1}{k-1+t}\phik{t}{k}$.
Hence
\smash{$\bigl( U + v \phik{t-1}{k-1+t} \bigr) \phik{t}{k} = 0$}. % \]
%and so \smash{$U + v \phik{t-1}{k-1+t} \in \ker \phik{t}{k}$}. 
Since 
\smash{$U + v \phik{t-1}{k-1+t} \in \F\Omegan{k}{n-1}$} and $n-1 \ge 2k + t$, 
applying the inductive hypothesis
to 
\[ \phikn{t}{k}{n-1}: \Omegan{k}{n-1} \longrightarrow \Omegan{k-t}{n-1} \]
gives
\begin{equation}
\label{eq:U} 
U + v \phik{t-1}{k-1+t} = W \phikn{t}{k+t}{n-1}
\end{equation} 
for some $W \in \F\Omegan{k+t}{n-1}$. Substituting for $U$ and $u$ using~\eqref{eq:u} and~\eqref{eq:U}
we find
\begin{align*}
 U + u \{n\} &= v \phik{t-1}{k-1+t} + W \phik{t}{k+t} + v \phik{t}{k-1+t} \{n\}\\
%             &= W \phik{t}{k+t}(n-1) + (v \{n\}) \phik{t}{k+t}(n) \\
             &= \bigl( W + v \{n\} \bigr) \phik{t}{k+t}, \end{align*}
hence $U + u \{n\} \in \im \phik{t}{k+t} : \F\Omega{k+t} \longrightarrow \F\Omega{k}$, as required.
\end{proof}

\section{Split exactness}
\label{sec:splitExact}

In this section we prove Theorem~\ref{thm:splitExact}, characterizing when the sequence

\vspace*{-18pt}
\[
\tag{2} \scalebox{0.975}{$\displaystyle 0 \rightarrow \F\Omega{a+ct} \xrightarrow{\phik{t}{a+ct}} \F\Omega{a+(c-1)t} 
\xrightarrow{\phik{t}{a+(c-1)t}} \cdots \xrightarrow{\phik{t}{a+2t}}
\F\Omega{a+t} \xrightarrow{\phik{t}{a+t}} \F\Omega{a} \rightarrow 0$} \]
is split exact. 
%if and only if $t$ is a two-power and $n \equiv 2a + t$ mod $2t$.
Suppose that there are just two non-zero modules. Then~\eqref{eq:phiComplex}~is
\[ 0 \rightarrow \F\Omega{a+t} \xrightarrow{\phik{t}{a+t}} \F\Omega{a} \rightarrow 0. \]
Comparing $\dim \F\Omega{a+t} = \binom{n}{a+t}$ and $\dim \F\Omega{a} = \binom{n}{a}$
shows that if \smash{$\phik{t}{a+t}$} is an isomorphism
then $n - (a+t) = a$, and so $n = 2a + t$, as required in condition~(a).
Since the chain complex is then self-dual,
Proposition~\ref{prop:surj} implies that \smash{$\phik{t}{a+t}$} is an isomorphism 
if and only if $a < 2^\tau$, where
$2^\tau$ is the least two-power appearing in the binary form of $a$. Hence
 condition (a) is necessary and sufficient for~\eqref{eq:phiComplex} to be split exact.

Now suppose~\eqref{eq:phiComplex} has at least three non-zero modules and is split exact. Therefore
condition~(a) does not hold.
If condition~(b) holds then $t = 2^\tau$ for some $\tau \in \N_0$ and
$n = 2a + (2s+1)2^\tau$ for some $s \in \N_0$. By maximality of $c$, we have $c = 2s+1$ and $n = 2a + ct$.
By Proposition~\ref{prop:surj}, $\phik{t}{a+t}$ is surjective
and, dually, $\phik{t}{a+ct}$ is injective. If $k = a + j2^\tau$ 
where $1 \le j < c$ then, since $n \equiv 2k + 2^\tau$ mod $2^{\tau+1}$, Proposition~\ref{prop:splitExact}
implies that \smash{$\F\Omega{k} = \ker \phik{t}{k} \oplus \im \phikstar{t}{k}$}. 
Hence~\eqref{eq:phiComplex} is split exact.
Conversely, suppose that~\eqref{eq:phiComplex}
%and 
 has at least three non-zero modules and is split exact.
Since it is then exact.
Theorem~\ref{thm:gen} implies that $t$ is a two-power.
Take $s$ maximal such that $2a + (2s+1)t \le n$
%Thus $a + 2^\tau s + s \le n - (a + 2^\tau s)$. 
and set $k = a +  (s+1)t$. The exact sequence
%By Theorem~\ref{thm:gen}, if
\[ \F\Omega{k+ t} \xrightarrow{\phik{t}{k+t}} \F\Omega{k}
\xrightarrow{\phik{t}{k}} \F\Omega{k-t} \]
is then part of~\eqref{eq:phiComplex}. By Theorem~\ref{thm:gen},
either $k+t \le n-k$ or $n-k+t \le k$. By choice of $s$ the first condition
does not hold. Therefore $n - \bigl( a + (s+1)t \bigr) + t \le a + (s+1)t$ and so
$n \le 2a +  (2s+1)t$. Hence $n = 2a + (2s+1)t$ and so $n \equiv 2a + t$ mod $2t$, as required in (b).
This completes the proof.

\section{Further directions}
\label{sec:problems}

Recall that $\gamma_k$ denotes $\phik{1}{k}$ and $\epsilon_k$ denotes~$\phik{2}{k}$.

\subsubsection*{Split exactness}
The sequence
$\F\Omega{k+t} \xrightarrow{\phik{t}{k+t}} \F\Omega{k} \xrightarrow{\phik{t}{k}} \F\Omega{k-t}$
in~\eqref{eq:seq}
was shown in Proposition~\ref{prop:splitExact} to be
 split exact when $t = 2^\tau$ is a two-power and $n \equiv 2k + 2^\tau$ mod $2^{\tau+1}$;
 call this condition (A).
By Propositions~\ref{prop:dual} and~\ref{prop:surj}
it is also split exact when  $k < t$ or $k > n-t$; call this condition (B).
% denote these two sufficient conditions. %, or equivalently, when there are just two non-zero modules.

If $t=1$ then the combined condition (A) or (B), namely that $n$ is odd or $k = 0$ or \hbox{$k=n$},
is necessary and sufficient for~\eqref{eq:seq} to be split exact. We outline a proof using 
that the ordinary character $\chi^{(n)} + \chi^{(n-1,1)} + \cdots + \chi^{(n-k,k)}$
of $\F \Omega{k}$ is multiplicity-free,
and so, by the results of \cite[\S 3.11]{Benson}, $\End_{\F S_n}(\F \Omega{k})$ is abelian.
It follows, by composing the projection maps, that if $V$ and~$W$ are distinct
direct summands of $\F \Omega{k}$ then $\Hom_{\F S_n}(V, W) = 0$. Hence the decomposition
of $\F \Omega{k}$ into direct summands is unique and each direct summand is self-dual.  
If $0 < k < n$ and~\eqref{eq:seq} splits then $\F \Omega{k}
\cong \ker \gamma_k \oplus C_k$ for some non-zero complement $C_k$.
We have $\im \gamma_k^\star \cong \mathrm{Ann} (\ker \gamma_k)
\cong C_k^\star \cong C_k$. Therefore there is an endomorphism of $\F \Omega{k}$ having
$\ker \gamma_k$ in its kernel, and restricting to an isomorphism $C_k \cong \im \gamma_k^\star$.
The uniqueness of the decomposition now shows that $\F \Omega{k} = \ker \gamma_k \oplus \im \gamma_k^\star$.
%Less obvious than it seems, e.g. can construct U + F -> F stuck at bottom of M, and then dual map
%kills F so don't get the complement in the expected way. Explicit example: take n=5 then
%	F Omega{1} -f-> F Omega{2} defined to kill S^{(4,1)} so 1 -> sum of all 2-subsets
%has dual map f* sending F Omega{2} to complement for S^{(4,1)}. But  ff* = 0. With
%duality can have im f* = C* /= C, so not a complement. (With permutation modules though ...)
However, by Lemma~\ref{lemma:phiphistar}, $\gamma_k \gamma_k^\star \not=0$
and $\gamma_k \gamma_k^\star + \gamma_{k+1}^\star \gamma_{k+1} = n\hskip0.5pt \mathrm{id}$, hence
$\gamma_k \gamma_k^\star \gamma_k = n \gamma_k$.
Therefore $\ker \gamma_k \cap \im \gamma_k^\star
\not= \{0\}$ whenever $n$ is even, showing that~\eqref{eq:seq} is not split in this case.

This argument can be adapted to show that, when $t=2$,~\eqref{eq:seq} is split if and only if
either (A) or (B) holds. Considerable calculation is required: for example, using only the
$\gamma$ and $\epsilon$ maps and their duals,
the simplest obstruction
to exactness  when $n \equiv 1$ mod~$4$ and $k$ is odd 
known to the author is
$\gamma_k^\star \epsilon_k \epsilon_k^\star \not= 0$ and 
$\gamma_k^\star \epsilon_k \epsilon_k^\star \epsilon_k = 0$.
 On the other hand, Example~\ref{ex:splitSurprise} shows that, when $t=4$, \eqref{eq:seq} may be split
in cases when neither (A) nor (B) holds. The following problem therefore appears to be quite deep.

\begin{problem}
Find a necessary and sufficient condition for~\eqref{eq:seq}
to be split exact.
\end{problem}

\subsubsection*{Generators for homology modules}

Recall that $G_\ell = \langle (1,2), \ldots, (2\ell-1, 2\ell) \rangle$.
Generalizing the elements $v_k$ defined before Theorem~\ref{thm:epsilonHomology}, we define
$\vk{t}{k} = \{2,4,\ldots, 2k\} \sum_{\sigma \in G_{k-t+1}} \sigma$.
By \cite[Theorem 17.13(i)]{James}, or a direct calculation similar to Lemma~\ref{lemma:vkdelta},
\smash{$\vk{t}{k}$} generates a submodule of \smash{$\ker \phik{t}{k}$}. 

\begin{conjecture}\label{conj:generation}
If $t$ is a two-power and $k \le 2n$ then the homology module
$\ker \phik{t}{k}/ \im \phik{t}{k+t}$ is generated by $\vk{t}{k} + \im \phik{t}{k+t}$.
\end{conjecture}

When $t=1$ the conjecture holds trivially because all the homology modules are zero.
 When $t=2$ it
is implied by Theorem~\ref{thm:epsilonHomology}.
It has been checked for all $n \le 16$ using {\sc Magma}
and the code available from the author's webpage.

\subsubsection*{Restricted homology}
Fix $s \in \N$. If $u \in \ker \phik{s}{k}$ then, by Lemma~\ref{lemma:composition},
$u \phik{t}{k} \phik{s}{k-t} %= \binom{s+t}{t} u \phik{t+s}{k} 
= u \phik{s}{k} \phik{t}{k-s} = 0$.
Therefore $\phik{t}{k} : \F\Omega{k} \rightarrow \F\Omega{k-t}$ restricts
to a map $\ker \phik{s}{k} \rightarrow \ker \phik{s}{k-t}$ and we may ask for the homology
of the sequence
\begin{equation}\label{eq:rseq} \ker \phik{s}{k+t}  
\xrightarrow{\phik{t}{k+t}} \ker \phik{s}{k} \xrightarrow{\phik{t}{k}} \ker \phik{s}{k-t}\hskip0.5pt. \end{equation}
%Calculations using the computer algebra package {\sc Magma} suggest that, when $t$ is a two-power,
%and the addition of $s$ to $t$ is carry free,~\eqref{eq:rseq} is exact for all sufficiently large $n$.
The following conjectures suggest that these restricted homology modules, denoted $\HR_k$, are surprisingly
well behaved. They have been checked for all $n \le 12$ using {\sc Magma} and the code available
from the author's webpage. 

\begin{conjecture}\label{conj:epsgammeEven}
Let $n = 2m$.
\begin{thmlist}
\item 
The sequence $\ker \gamma_{k+2} \xrightarrow{\epsilon_{k+2}} \ker \gamma_{k} \xrightarrow{\epsilon_k}
\ker \gamma_{k-2}$ has non-zero homology if and only if $k \in \{m-1,m\}$. Moreover $\HR_{m-1} \cong \HR_m 
\cong D^{(m+1,m-1)}$.
\item The sequence $\ker \epsilon_{k+1} \xrightarrow{\gamma_{k+1}} \ker \epsilon_{k} \xrightarrow{\gamma_k}
\ker \epsilon_{k-1}$ has non-zero homology if and only if $k = m$. Moreover $\HR_m \cong D^{(m+1,m-1)}$.
\end{thmlist}
\end{conjecture}

\begin{conjecture}\label{conj:epsgammeOdd}
Let $n = 2m+1$.
\begin{thmlist}
\item 
The sequence $\ker \gamma_{k+2} \xrightarrow{\epsilon_{k+2}} \ker \gamma_{k} \xrightarrow{\epsilon_k}
\ker \gamma_{k-2}$ has non-zero homology if and only if $k = m$. Moreover $\HR_m \cong D^{(m+1,m)}$.
\item The sequence $\ker \epsilon_{k+1} \xrightarrow{\gamma_{k+1}} \ker \epsilon_{k} \xrightarrow{\gamma_k}
\ker \epsilon_{k-1}$ is exact.
\end{thmlist}
\end{conjecture}

%\begin{conjecture}\label{conj:gammaOnKerEps}
%The sequence $\ker \epsilon_{k+1} \xrightarrow{\gamma_{k+1}} \ker \epsilon_k \xrightarrow{\gamma_k}
%\ker \epsilon_{k-1}$ has non-zero homology if and only if $n = 2m$ is even and $k = m$. Moreover
%$H_m \cong D^{(m+1,m-1)}$.
%\end{conjecture}

For example, taking $n=6$ as in 
Example~\ref{ex:ex},
the chain complex with restricted maps
\smash{$0 \rightarrow \ker \gamma_{6} \xrightarrow{\epsilon_6} \ker \gamma_{4} \xrightarrow{\epsilon_4} \ker \gamma_{2}
\xrightarrow{\epsilon_2} \ker \gamma_{0} \rightarrow 0$} is
%\smash{$\ker \phik{1}{6} \xrightarrow{\epsilon_6} \ker \phik{1}{4} \xrightarrow{\epsilon_4} \ker \phik{1}{2}
%\xrightarrow{\epsilon_2} \ker \phik{1}{0}$} is
\[ 0 \rightarrow 0 \xrightarrow{\epsilon_6} \begin{matrix} \F \\ D^{(5,1)} \end{matrix}
\xrightarrow{\epsilon_4} \begin{matrix} \F \\ D^{(4,2)} \\[-3pt] \F \\ D^{(5,1)} \end{matrix}
\xrightarrow{\epsilon_2} \F \rightarrow 0\]
which has non-zero homology of $D^{(4,2)}$ uniquely in degree $2$. This chain complex is dual
to the chain complex 
\smash{$0 \rightarrow \ker \gamma_{5} \xrightarrow{\epsilon_5} \ker \gamma_{3} \xrightarrow{\epsilon_3} \ker \gamma_{1}
\xrightarrow{\epsilon_1} 0$} which has non-zero homology of $D^{(4,2)}$ uniquely in degree $3$.
The chain complex
$0\rightarrow \ker \epsilon_6 \xrightarrow{\gamma_6} \ker \epsilon_5 \xrightarrow{\gamma_5} \cdots 
\xrightarrow{\gamma_2}\ker \epsilon_1 \xrightarrow{\gamma_1} \ker \epsilon_0 \rightarrow 0$ is

\hspace*{-0.1in}\begin{tikzpicture}[x=0.5cm, y=0.5cm]
\node[right] (O) at (0,0) {
 $0 \rightarrow 0 \xrightarrow{\,\gamma_6\,} 0 \xrightarrow{\,\gamma_5\,} \F \xrightarrow{\,\gamma_4\,} \begin{matrix}
D^{(4,2)} \\ \F \\ D^{(5,1)} \oplus\; D^{(4,2)} \\ \F \end{matrix} \xrightarrow{\,\gamma_3\,}\; \begin{matrix} D^{(5,1)} \\
\F \\ D^{(4,2)} \\ \F \\ D^{(5,1)} \end{matrix}\; \xrightarrow{\,\gamma_2\,} \begin{matrix} \F \\ D^{(5,1)} \\ \F 
\end{matrix}\; \xrightarrow{\,\gamma_1\,}\; \F \rightarrow 0$};  

\draw (11.1,0.2)--(13.2,0.2)--(13.2,-1)--(11.5,-1)--(11.5,-2.25)--(10.1,-2.25)--(10.1,-1)--(11.1,-1)--(11.1,0.2);
\draw(14.8,-2.5)--(17.2,-2.5)--(17.2,0.6)--(14.8,0.6)--(14.8,-2.5);
\draw(18.8,-1.5)--(21.0,-1.5)--(21.0,0.6)--(18.8,0.6)--(18.8,-1.5);
\draw(22.7,-0.5)--(23.7,-0.5)--(23.7,0.6)--(22.7,0.6)--(22.7,-0.5);
\end{tikzpicture}

\noindent 
where the boxes show the kernels of the maps $\gamma_k$, now each
restricted to $\ker \epsilon_k$. It has non-zero homology of $D^{(4,2)}$ uniquely in degree $3$.

\subsubsection*{Multistep maps in odd characteristic}
Now suppose that $\F$ has odd prime characteristic $p$.
Lemma~\ref{lemma:composition} generalizes
to show that $\phik{s}{k+s} \phik{t}{k} = 0$ whenever~$p$ divides $\binom{s+t}{s}$.
(Equivalently, a carry arises when $s$ and $t$ are added in base~$p$.)
Generalizing the usual definition, we may ask for the homology 
$H_k = \ker \phik{t}{k} / \im \phik{s}{k+s}$ of the sequence
\begin{equation}
\label{eq:genseq}
\F\Omega{k+s} \xrightarrow{\phik{s}{k+s}} \F\Omega{k} \xrightarrow{\phik{t}{k}} \F\Omega{k-t}.\end{equation}
The following two conjectures have been checked for all $n \le 12$ using {\sc Magma}
and the code available from the author's webpage.

\begin{conjecture}\label{conj:odd3}
If $p=3$ then 
$\F\Omega{k+2} \xrightarrow{\epsilon_{k+2}} \F\Omega{k} \xrightarrow{\gamma_{k}}
\F\Omega{k-1}$
%$\F\Omega{k+2} \xrightarrow{\phik{2}{k+2}} \F\Omega{k} \xrightarrow{\phik{1}{k}}
%\F\Omega{k-1}$ 
has non-zero homology if and only if $k = \lfloor n/2 \rfloor$. Moreover
in the exceptional case $H_k$ is isomorphic to the sign module.
\end{conjecture}
 
Taking $n=2m$, James' $p$-regularization theorem (see \cite{JamesDecII})
implies that  $\sgn \cong D^{(m,m)}$ when $\F$ has characteristic $3$.
The analogue of Proposition~\ref{prop:cfs} then
implies that $\sgn$ is a composition factor of $\F\Omega{m}$, 
but not of either $\F\Omega{m+1}$ or $\F\Omega{m-2}$.
Hence $H_m$ has the sign module as a composition factor. By the 
argument seen in the proof of Corollary~\ref{cor:dimensions}, a proof of Conjecture~\ref{conj:odd3}
will categorify the binomial identity
\begin{equation}
\label{eq:p3} \sum_{j} \binom{n}{3j} - \sum_{j} \binom{n}{3j+1} = \begin{cases} (-1)^n & \text{if $n \equiv 0$ mod $3$}
\\ 0 & \text{if $n \equiv 1$ mod $3$} \\ (-1)^{n-1} & \text{if $n \equiv 2$ mod $3$.} \end{cases} \end{equation}
(This identity follows at once from (6.14) and (6.22) in \cite{GouldVol6}, or by adapting
the proof of~\eqref{eq:binom} in~\S\ref{sec:epsilonHomology}, or most easily, by induction on $n$.)
%\S\ref{sec:epsilonHomology}. It also 
%can be proved by induction on $n$.)
For example, when $n=10$ the identity is categorified by the chain complex 
%\[ \F\Omega{10} \rightarrow \F\Omega{9} \rightarrow \F\Omega{7} \rightarrow \F\Omega{6} \rightarrow
%\F\Omega{4} \rightarrow \F\Omega{3} \rightarrow \F\Omega{1} \rightarrow \F\Omega{0} \]
%\[ \F\Omega{10} \xrightarrow{\phik{1}{10}} \F\Omega{9} \xrightarrow{\phik{2}{9}} \F\Omega{7} \xrightarrow{\phik{1}{7}}  \F\Omega{6} \xrightarrow{\phik{2}{6}} 
%\F\Omega{4} \xrightarrow{\phik{1}{4}}  \F\Omega{3} \xrightarrow{\phik{2}{3}}  \F\Omega{1} 
%\xrightarrow{\phik{1}{1}}  \F\Omega{0} \]
\[ 0 \rightarrow \F\Omega{10} \xrightarrow{\gamma_{10}} \F\Omega{9} \xrightarrow{\epsilon_{9}} \F\Omega{7} \xrightarrow{\gamma_{7}}  \F\Omega{6} \xrightarrow{\epsilon_{6}} 
\F\Omega{4} \xrightarrow{\gamma_{4}}  \F\Omega{3} \xrightarrow{\epsilon_{3}}  \F\Omega{1} 
\xrightarrow{\gamma_{1}}  \F\Omega{0} \rightarrow 0,\]
which is exact in every degree. 

\begin{conjecture}\label{conj:odd5}
If $p=5$ then $\F\Omega{k+4} \xrightarrow{\phik{4}{k+4}} \F\Omega{k} \xrightarrow{\gamma_{k}}
\F\Omega{k-1}$ has non-zero homology if and only if $k 
\in \{ \lfloor n/2 \rfloor,  \lfloor n/2 \rfloor -1\}$.
Moreover, if $n=2m$ is even then $H_{m-1} \cong D^{(m+1,m-1)}$ and $H_m \cong D^{(m,m)}$,
and if $n=2m+1$ is odd then $H_{m-1} \cong D^{(m+2,m-1)}$ and $H_m \cong D^{(m+1,m)}$.
\end{conjecture}

Again it is straightforward to show that the homology modules have the specified
simple modules as composition factors.
Somewhat remarkably, the dimensions of these simple modules appear to be
certain Fibonacci numbers, as defined
by $F_0 = 0$, $F_1 = 1$ and $F_n = F_{n-1} + F_{n-2}$ for $n\ge 2$.
A proof of Conjecture~\ref{conj:odd5} will imply
that $\dim D^{(m,m)} = F_{2m-1}$ and $\dim D^{(m+1,m-1)}= \dim D^{(m+2,m-1)} = F_{2m}$,
and categorify a family of binomial identities
including 
\begin{equation}\label{eq:Andrews1}
\sum_{j} \binom{5m}{5j} - \sum_{j} \binom{5m}{5j+1} =  (-1)^{m}F_{5m-1} 
\end{equation} 
and $\sum_{j} \binom{5m+2}{5j} - \sum_{j} \binom{5m+2}{5j+1} = (-1)^{m-1}F_{5m+1} $.
These identities are somewhat deeper than~\eqref{eq:p3}. Taken together they are equivalent to the identity
\begin{equation}\label{eq:Andrews} F_n = \sum_k (-1)^k \binom{n}{\lfloor \frac{n-1 - 5k}{2} \rfloor} \end{equation}
proved by Andrews in \cite{Andrews} and later, with a simpler inductive proof, by Gupta in \cite{GuptaFibonacci}.
For example, since  %when $n = 10m-1$ we have
$\lfloor \frac{10r-2-5k}{2} \rfloor \equiv (-1)^{k-1}$ mod~$5$,
Andrews' identity implies that $F_{10r-1} = \sum_j \binom{10r-1}{5j-1} - \sum_j \binom{10r-1}{5j+1}$.
Since $\binom{5m}{5j} = \binom{5m-1}{5j} + \binom{5m-1}{5j-1}$ and $\binom{5m}{5j+1} = \binom{5m-1}{5j+1}
+\binom{5m-1}{5j}$,
this is equivalent to~\eqref{eq:Andrews1} when $m$ is even.

%\renewcommand{\MR}[1]{\relax}
%\bibliographystyle{amsplain}
%\bibliography{References}

\begin{thebibliography}{10}

\enlargethispage{2pt}
\bibitem{Andrews}
G.~E. Andrews, \emph{Two theorems of {G}auss and allied identities proved
  arithmetically}, Pacific J. Math. \textbf{41} (1972), 563--578. \MR{MR0349572
  (50 \#2065)}

\bibitem{Benson}
D.~J. Benson, \emph{Representations and cohomology. {I}}, second ed., Cambridge
  Studies in Advanced Mathematics, vol.~30, Cambridge University Press,
  Cambridge, 1998, Basic representation theory of finite groups and associative
  algebras. \MR{MR1644252 (99f:20001a)}

\bibitem{BensonSpin}
Dave Benson, \emph{Spin modules for symmetric groups}, J. London Math. Soc. (2)
  \textbf{38} (1988), no.~2, 250--262. \MR{966297}

\bibitem{BoltjeHartmann}
Robert Boltje and Robert Hartmann, \emph{Permutation resolutions for {S}pecht
  modules}, J. Algebraic Combin. \textbf{34} (2011), no.~1, 141--162.
  \MR{2805203}

\bibitem{Magma}
Wieb Bosma, John Cannon, and Catherine Playoust, \emph{The {M}agma algebra
  system. {I}. {T}he user language}, J. Symbolic Comput. \textbf{24} (1997),
  no.~3-4, 235--265, Computational algebra and number theory (London, 1993).
  \MR{MR1484478}

\bibitem{DanzKulshammer}
Susanne Danz and Burkhard K\"ulshammer, \emph{The vertices and sources of the
  basic spin module for the symmetric group in characteristic 2}, J. Pure Appl.
  Algebra \textbf{213} (2009), no.~7, 1264--1282. \MR{2497574}

\bibitem{DotyErdmannHenkeYoung}
Stephen Doty, Karin Erdmann, and Anne Henke, \emph{Endomorphism rings of
  permutation modules over maximal {Y}oung subgroups}, J. Algebra \textbf{307}
  (2007), no.~1, 377--396. \MR{2278061}

\bibitem{GiannelliLimWildon}
Eugenio Giannelli, Kay~Jin Lim, and Mark Wildon, \emph{Sylow subgroups of
  symmetric and alternating groups and the vertex of {${S}^{(kp-p,1^p)}$} in
  characteristic {$p$}}, J. Algebra \textbf{455} (2016), 358--385. \MR{3478866}

\bibitem{GouldVol6}
Henry~W. Gould, \emph{Tables of combinatorial identities},
  \url{http://www.math.wvu.edu/~gould/Vol.6.PDF} \textbf{6} (May 2010), 26
  pages.

\bibitem{GuptaFibonacci}
Hansraj Gupta, \emph{The {A}ndrews formula for {F}ibonacci numbers}, Fibonacci
  Quart. \textbf{16} (1978), no.~6, 552--555. \MR{515986}

\bibitem{Hamernik}
Wolfgang Hamernik, \emph{{S}pecht modules and the radical of the group ring
  over the symmetric group {$\gamma_p$}}, Comm.~Algebra \textbf{4}
  (1976), 435--475.

\bibitem{HenkeYoung}
Anne Henke, \emph{On {$p$}-{K}ostka numbers and {Y}oung modules}, European J.
  Combin. \textbf{26} (2005), no.~6, 923--942. \MR{2143202}

\bibitem{JK}
G.~James and A.~Kerber, \emph{The representation theory of the symmetric
  group}, Encyclopedia of Mathematics and its Applications, vol.~16,
  Addison-Wesley Publishing Co., Reading, Mass., 1981. \MR{MR644144
  (83k:20003)}

\bibitem{JamesDecII}
G.~D. James, \emph{On the decomposition matrices of the symmetric groups.
  {II}}, J. Algebra \textbf{43} (1976), no.~1, 45--54. \MR{0430050 (55
  \#3057b)}

\bibitem{JamesChar2}
\bysame, \emph{Representations of the symmetric groups over the field of order
  {$2$}}, J. Algebra \textbf{38} (1976), no.~2, 280--308. \MR{0396734 (53
  \#595)}

\bibitem{James}
\bysame, \emph{The representation theory of the symmetric groups}, Lecture
  Notes in Mathematics, vol. 682, Springer, Berlin, 1978. \MR{MR513828
  (80g:20019)}

\bibitem{KleshchevBook}
Alexander Kleshchev, \emph{Linear and projective representations of symmetric
  groups}, Cambridge Tracts in Mathematics, vol. 163, Cambridge University
  Press, Cambridge, 2005. \MR{2165457}

\bibitem{MullerZimmermann}
J{\"u}rgen M{\"u}ller and Ren{\'e} Zimmermann, \emph{Green vertices and sources
  of simple modules of the symmetric group labelled by hook partitions}, Arch.
  Math. (Basel) \textbf{89} (2007), no.~2, 97--108. \MR{2341720 (2008g:20020)}

\bibitem{MurphyDecomposable}
G.~M. Murphy, \emph{On decomposability of some {S}pecht modules for symmetric
  groups}, J. Algebra \textbf{66} (1980), no.~1, 156--168.

\bibitem{PeelHooks}
M.~H. Peel, \emph{Hook representations of the symmetric groups}, Glasgow Math.
  J. \textbf{12} (1971), 136--149.

\bibitem{SantanaYudin}
Ana~Paula Santana and Ivan Yudin, \emph{Characteristic-free resolutions of
  {W}eyl and {S}pecht modules}, Adv. Math. \textbf{229} (2012), no.~4,
  2578--2601. \MR{2880232}

\bibitem{WildonThesis}
Mark Wildon, \emph{Modular representations of symmetric groups}, D.~Phil. thesis,
  Oxford University, 2004.

\end{thebibliography}

\def\cprime{$'$} \def\Dbar{\leavevmode\lower.6ex\hbox to 0pt{\hskip-.23ex
  \accent"16\hss}D} \def\cprime{$'$}
\providecommand{\bysame}{\leavevmode\hbox to3em{\hrulefill}\thinspace}
\providecommand{\MR}{\relax\ifhmode\unskip\space\fi MR }
\providecommand{\MRhref}[2]{%
  \href{http://www.ams.org/mathscinet-getitem?mr=#1}{#2}
}
\providecommand{\href}[2]{#2}
\renewcommand{\MR}[1]{\relax}

\end{document}